\documentclass{elsarticle}

\usepackage{amsmath,amssymb,amsopn,amsthm}
\usepackage{url}
\usepackage{shortvrb}
\usepackage{stmaryrd} % extra symbols
\usepackage{graphicx}
\usepackage[caption=false]{subfig}
\usepackage{booktabs}
\usepackage{enumitem}
\usepackage{hyperref}
\usepackage[T1]{fontenc}

\usepackage[colorinlistoftodos,prependcaption,textsize=footnotesize]{todonotes}
\reversemarginpar

\usepackage[normalem]{ulem}

\makeatletter
\g@addto@macro\bfseries{\boldmath}
\makeatother

\newtheorem{theorem}{Theorem}[section]
\newtheorem{remark}[theorem]{Remark}
\newtheorem{prop}[theorem]{Proposition}
\newtheorem{lemma}[theorem]{Lemma}

\numberwithin{equation}{section}

\newcommand{\R}{\mathbb{R}}

\DeclareMathOperator*{\diam}{diam}

\newcommand{\mcK}{\mathcal{K}}

\newcommand{\mcN}{\mathcal{N}}

\newcommand{\mcO}{\mathcal{O}}

\newcommand{\tn}{|\mspace{-1mu}|\mspace{-1mu}|}

\newcommand{\hatOmega}{\widehat{\Omega}}
\newcommand{\hatmcK}{\widehat{\mcK}}

\newcommand{\OO}{\mathcal{O}}

\newcommand{\hatv}{\widehat{v}}
\newcommand{\hatw}{\widehat{w}}
\newcommand{\hatA}{\widehat{A}}

\newcommand{\hmax}{h}

\begin{document}

%---------------------------------------------------------------------------
\begin{frontmatter}
\title{MultiMesh Finite Elements with Flexible Mesh Sizes}
\author[sintef]{August Johansson\corref{cor1}}
\ead{august.johansson@sintef.no}
\author[umu]{Mats G. Larson}
\ead{mats.larson@umu.se}
\author[cth]{Anders Logg}
\ead{logg@chalmers.se}
\cortext[cor1]{Corresponding author}
\address[sintef]{SINTEF Digital, Mathematics and Cybernetics, P.O.\ Box 124 Blindern, 0314 Oslo, Norway}
\address[umu]{Department of Mathematics and Mathematical Statistics, Ume{\aa} University, 90187 Ume{\aa}, Sweden}
\address[cth]{Department of Mathematical Sciences, Chalmers University of Technology and University of Gothenburg, 41296 G\"oteborg, Sweden}

\begin{abstract}
We analyze a new framework for expressing finite element methods on arbitrarily many intersecting meshes: multimesh finite element methods.
The multimesh finite element method, first presented in~\cite{mmfem-1},
enables the use of separate meshes to discretize parts of a computational domain that are naturally separate; such as the componen{}ts of an engine, the  domains of a multiphysics problem, or solid bodies interacting under the influence of forces from surrounding fluids or other physical fields. Furthermore, each of these meshes may have its own mesh parameter.

In the present paper we study the Poisson equation and show that the proposed formulation is stable without assumptions on the relative sizes of the mesh parameters.
In particular, we prove optimal order \emph{a~priori} error estimates as well as optimal order estimates of the condition number. Throughout the analysis, we trace the dependence of the number of intersecting meshes. Numerical examples are included to illustrate the stability of the method.
\end{abstract}


%\begin{AMS}
%  65N30, 65N85, 65Y99, 68U20
%\end{AMS}
\end{frontmatter}

%---------------------------------------------------------------------------
\section{Introduction}

The multimesh finite element method presented in~\cite{mmfem-1} extends the finite element method to arbitrarily many overlapping and intersecting meshes. This is of great value for problems that are naturally formulated on domains composed of \emph{parts}, such as complex domains composed of simpler parts that may be more easily meshed than their composition. This is of particular importance when the parts are moving, either relative to each other or relative to a fixed background mesh, as part of a time-dependent simulation or optimization problem~\cite{Dokken:2017aa,DokkenNS}. Figure~\ref{fig:motivation} provides some illustrative examples. Here, as in~\cite{mmfem-1}, we consider the Poisson equation with stationary interfaces to simplify the analysis.

The mathematical basis for the multimesh element method is Nitsche's method \cite{Nitsche:1971aa}, which is here used for weakly enforcing the interface conditions between the different meshes. Nitsche's method is also the basis for discontinuous Galerkin methods~\cite{Arnold:1982aa} which also may be cast in a setting of non-matching meshes~\cite{NME:NME2631,Johansson:2013:HOD:2716605.2716805,GURKAN2019466,SAYE2017647,SAYE2017683}. In addition, Nitsche's method is also the foundation of the finite element method on cut meshes, CutFEM, see for example~\cite{Hansbo:2002aa,Hansbo:2003aa,Becker:2009aa,Burman2012,Burman2017,claus2018} or~\cite{burman:2015aa,Bordas-et-al-2017} for overviews.

Several methods for treating the interface problems with non-matching and multiple meshes exist in literature. There are techniques based on XFEM~\cite{Moes1999131,NME:NME2914,NME:NME386,LEHRENFELD2016716,FORMAGGIA2018893,ALAUZET2016300,Zonca2018,Moumnassi2014,MOUMNASSI2011774}; domain decomposition~\cite{RANK1992299,BeckerHansboStenberg} or~\cite{APPELO20126012,Henshaw:2008:PCT:1387360.1387531} and the references therein; the finite cell method~\cite{Schillinger2015,DUSTER20083768,Parvizian2007,Xu2016,DEPRENTER2017297,prenter2}; the immersed interface method~\cite{LI1998253,li2006immersed}; the classical immersed boundary methods and its variants using finite elements~\cite{Peskin,Boffi2003491,Boffi20152584,HELTAI2012110,Lui2007,Zhang20042051}; the s-version of the finite element method~\cite{FISH1992539,FISH1994135} and fictitious domain methods~\cite{GLOWINSKI1994283,nagai2007,KADAPA20161}, to name a few.

Another approach is to use a matching mesh and make use of elements with polytopic shapes. Methods with this capability include the PolyDG method~\cite{cangiani2014,Antonietti2016}, hybrid high order methods~\cite{DiPietroErn2015,DiPietroErnLemaire2016}, virtual element methods~\cite{BeiraoBrezziCangianiEtAl2013,BeiraodaVeigaBrezziMariniEtAl2014}, and mimetic methods~\cite{BrezziLipnikovSimoncini2005,LipnikovManziniShashkov2014}.

The contributions of this paper is first a generalization of the formulation in~\cite{mmfem-1} for the Poisson equation to allow for meshes of arbitrary mesh sizes. In the formulation, each mesh has its own mesh size and can be placed in a general position. The
properties of the mesh arrangement is encoded in terms of the maximum number of overlapping meshes at any point in the domain. Naturally, this number may be much lower than the total number of meshes. The second contribution is a detailed analysis of the method. We carefully trace the dependency of the number of intersecting meshes in the coercivity of the method, in the error estimates and in the analysis of the condition number. The analysis holds for two and three dimension as well as for higher order elements, and extends previous works on cut finite elements for overlapping meshes and interface problems to much more general mesh arrangements and mesh sizes. We restrict ourselves to two dimensions in the numerical examples. See also~\cite{Massing:2014aa,Johansson:2015aa,johansson2018stokes} where related formulations for the Stokes problem are presented and analyzed.

In the remainder of this paper, we analyze the multimesh finite element method for the Poisson problem for an arbitrary number of intersecting meshes and arbitrarily mesh sizes, and present numerical examples. We will start with reviewing the notation from~\cite{mmfem-1} in Section~\ref{sec:notation}, following a presentation of the multimesh finite element method~\ref{sec:formulation}. We then proceed to establish standard results such as consistency and continuity of the method in Section~\ref{sec:fem_basics}. Showing coercivity, interpolation error estimates, \emph{a priori} error estimates and a condition number estimate require more work, which is why we dedicate individual Sections to these in~\ref{sec:coercive},~\ref{sec:interpolationerrest},~\ref{sec:apriori}, and~\ref{sec:conditioning} correspondingly. We end the paper with numerical results in Section~\ref{sec:numres}, conclusions in Section~\ref{sec:conclusions} and acknowledgments in Section~\ref{sec:acks}.

\begin{figure}[htbp]
  \centering{}
  \includegraphics[width=0.4\textwidth]{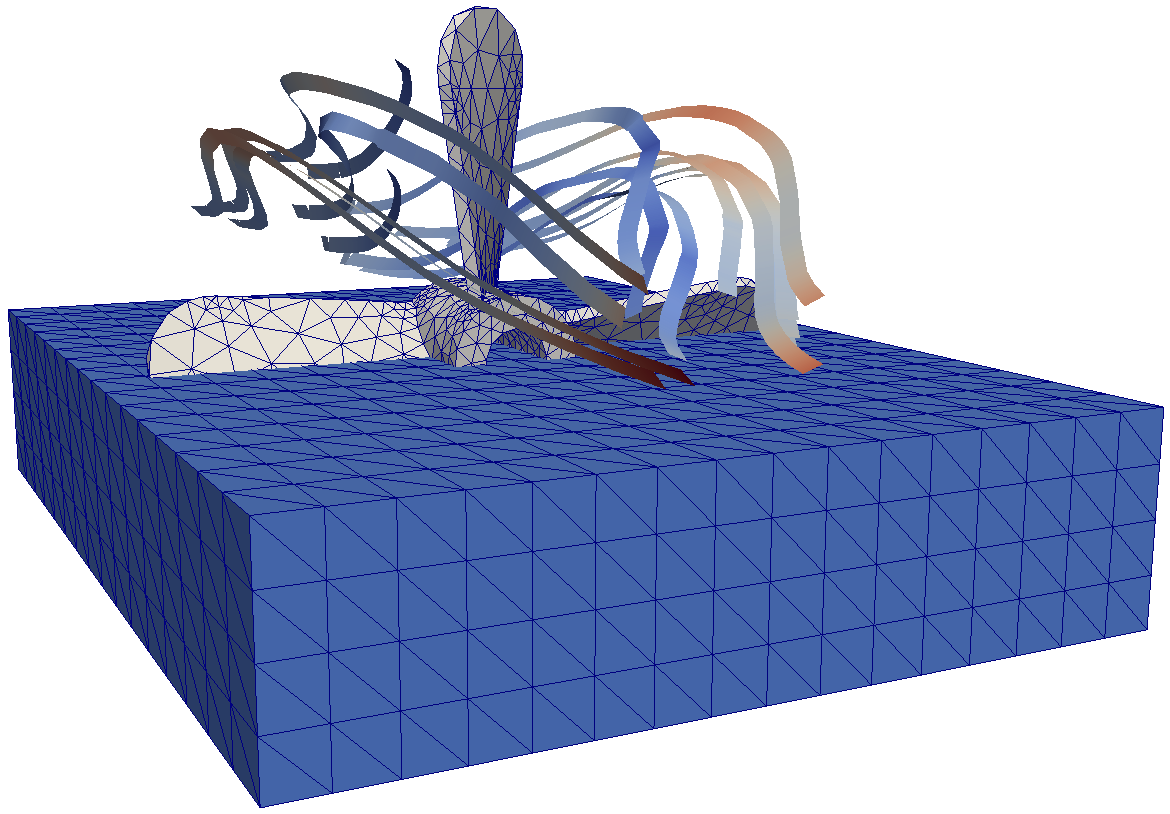} \quad
  \includegraphics[width=0.4\textwidth]{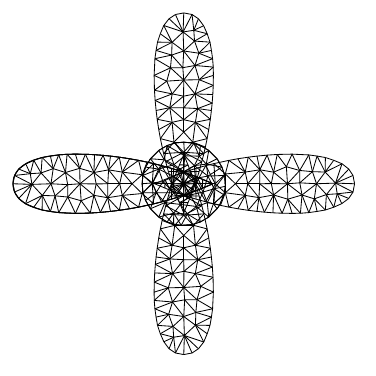} \\[0.5em]
  \includegraphics[width=0.85\textwidth]{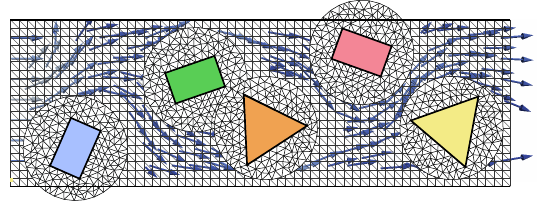}
  \caption{(Top left) The flow around a propeller may be computed by immersing a mesh of the propeller
into a fixed background mesh. (Top right) The geometry of a composite object may be discretized by superimposing meshes of each component. (Bottom) The interaction of a set of solid bodies may be simulated
using individual meshes that move and intersect freely relative to each other and a fixed background mesh. (These illustrations also appear in~\cite{mmfem-1}.) }
  \label{fig:motivation}
\end{figure}

%---------------------------------------------------------------------------
\section{Notation}
\label{sec:notation}

We first review the notation for domains, interfaces, meshes, overlaps, function spaces and norms used to formulate and analyze the multimesh finite element method. For a more detailed exposition, we refer to~\cite{mmfem-1}.

\subsection{Notation for domains}
\begin{itemize}[noitemsep,topsep=0pt,parsep=0pt,partopsep=0pt]
\item Let $\Omega = \hatOmega_0 \subset \R^d$, $d = 2,3$, be a domain with polygonal boundary (the background domain).
\item Let $\hatOmega_i \subset \hatOmega_0$, $i=1,\ldots, N$ be the so-called \emph{predomains} with polygonal boundaries (see Figure~\ref{fig:three_domains}). Note that these are placed in an ordering.
\item Let $\Omega_i = \hatOmega_i \setminus \bigcup_{j=i+1}^{N} \hatOmega_j$,  $i=0, \ldots, N$ be a partition of $\Omega$ (see Figure~\ref{fig:three_domains_partition}). Note that
this means that $x \in \Omega$ belongs to $\Omega_i$, where $i$ is the largest index $j$
such that $x\in \hatOmega_j$, i.e.,
\begin{equation}
i = \max \{ j : x \in \hatOmega_j \}.
\end{equation}
\end{itemize}

\begin{figure}
  \centering
  \subfloat[]{\label{fig:three_domains_a}\includegraphics[height=0.25\linewidth]{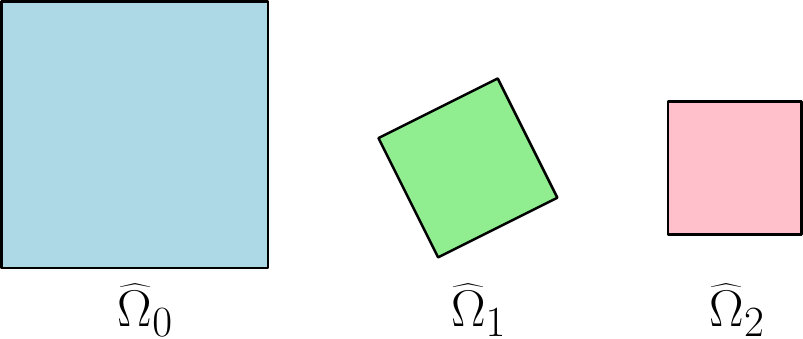}}\qquad\qquad\qquad
  \subfloat[]{\label{fig:three_domains_b}\includegraphics[height=0.25\linewidth]{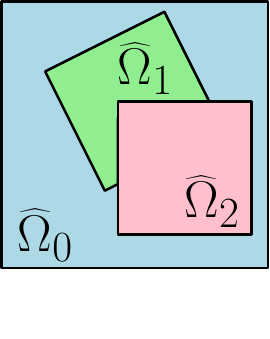}}\\
  \subfloat[]{\label{fig:three_domains_partition}\includegraphics[height=0.25\linewidth]{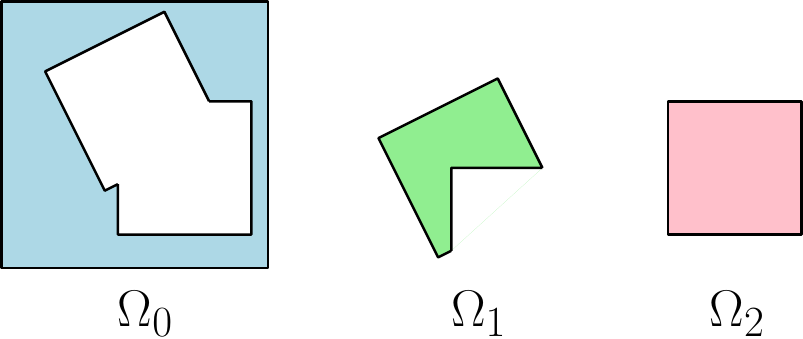}}
  \caption{(a) Three polygonal predomains. (b) The predomains are placed on top of each other in an ordering such that
    $\hatOmega_0$ is placed lowest, $\hatOmega_1$ is in the middle and $\hatOmega_2$ is on top. (c) Partition of $\Omega = \Omega_0 \cup \Omega_1 \cup \Omega_2$. Note that $\Omega_2 = \hatOmega_2$. (These illustrations also appear in~\cite{mmfem-1}.)}
  \label{fig:three_domains}
\end{figure}

\begin{remark}
  \label{rem:boundary-overlap}
To simplify the presentation, the domains $\Omega_1, \ldots, \Omega_N$ are not allowed to intersect the boundary of $\Omega$. The method can be extended to include situations where
the subdomains may intersect the boundary by using weak enforcement of boundary conditions. If cut elements appear at the boundary some stabilization of the formulation is
needed, for instance face based least squares control of jumps in derivatives across faces in the vicinity of the boundary \cite{burman:2015aa} or an extension procedure \cite{johanssonlarson2013}.
\end{remark}

\subsection{Notation for interfaces}
\label{sec:interfaces}
\begin{itemize}[noitemsep,topsep=0pt,parsep=0pt,partopsep=0pt]
\item Let the \emph{interface} $\Gamma_i$ be defined by $\Gamma_i = \partial \hatOmega_i \setminus \bigcup_{j=i+1}^N \hatOmega_j$, $i=1, \ldots, N-1$ (see Figure~\ref{fig:two_interfaces_a}).
\item Let $\Gamma_{ij} = \Gamma_i \cap \Omega_j$, $i > j$, be a partition of $\Gamma_i$ (see Figure~\ref{fig:two_interfaces_b}).
\end{itemize}

\begin{figure}
  \centering
  \subfloat[]{\label{fig:two_interfaces_a}\includegraphics[width=0.25\linewidth]{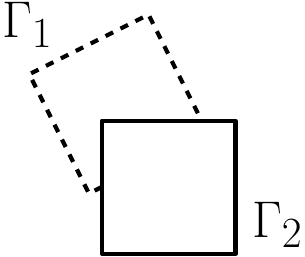}}\qquad\qquad\qquad
  \subfloat[]{\label{fig:two_interfaces_b}\includegraphics[width=0.25\linewidth]{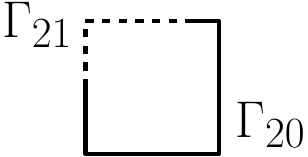}}
  \caption{(a) The two interfaces of the domains in Figure~\ref{fig:three_domains}: $\Gamma_1 = \partial \hatOmega_1 \setminus \hatOmega_2$ (dashed line) and  $\Gamma_2 = \partial \hatOmega_2$ (filled line). Note that $\Gamma_1$ is not a closed curve. (b) Partition of $\Gamma_2 = \Gamma_{20} \cup \Gamma_{21}$. (These illustrations also appear in~\cite{mmfem-1}.)} 
\end{figure}

\subsection{Notation for meshes}
\begin{itemize}[noitemsep,topsep=0pt,parsep=0pt,partopsep=0pt]
\item Let $\hatmcK_{h,i}$ be a quasi-uniform \cite{BreSco08} \emph{premesh} on $\hatOmega_i$ with mesh parameter $h_i = \max_{K\in \hatmcK_{h,i}} \diam(K)$, $i=0,\ldots,N$ (see Figure~\ref{fig:three_meshes_a}).
\item Let $\displaystyle \hmax = \max_{0\leq i \leq N} h_i$.
\item Let $\mcK_{h,i} = \{ K \in \hatmcK_{h,i} : K \cap \Omega_i \neq \emptyset \}$, $i=0,\ldots,N$ be the \emph{active meshes} (see Figure~\ref{fig:three_meshes_b}).
\item The \emph{multimesh} is formed by the active meshes placed in the given ordering (see Figure~\ref{fig:multimesh}).
\item Let $\Omega_{h,i} = \bigcup_{K\in\mcK_{h,i}} K$, $i=0,\ldots,N$ be the \emph{active domains}.
\end{itemize}

\begin{figure}
  \centering
  \subfloat[]{\label{fig:three_meshes_a}\includegraphics[height=0.25\linewidth]{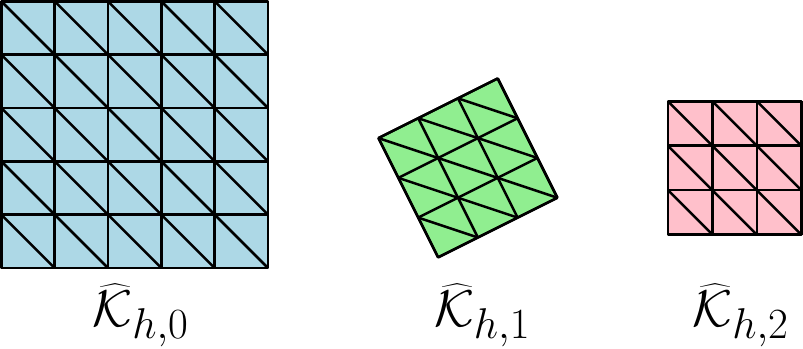}}\qquad\qquad\qquad
  \subfloat[]{\label{fig:three_meshes_b}\includegraphics[height=0.25\linewidth]{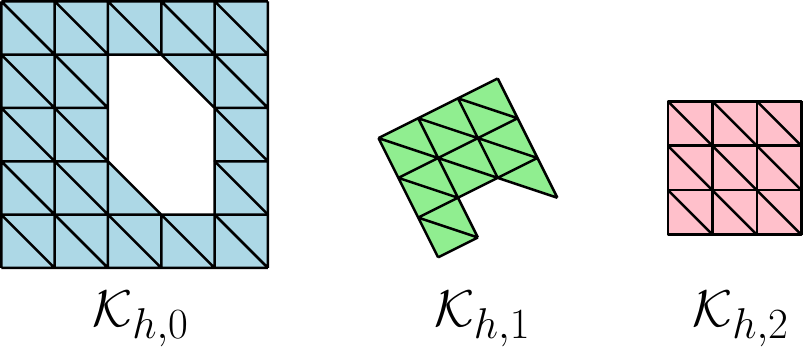}}
  \caption{(a) The three premeshes. (b) The corresponding active meshes (cf.~Figures~\ref{fig:three_domains} and~\ref{fig:three_domains_partition}). (These illustrations also appear in~\cite{mmfem-1}.)}
\end{figure}

\begin{figure}
  \centering
  \subfloat[]{\label{fig:overlap}\includegraphics[height=0.25\linewidth]{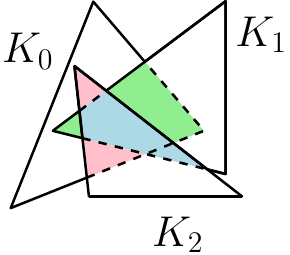}}\qquad\qquad
  \subfloat[]{\label{fig:multimesh}\includegraphics[height=0.25\linewidth]{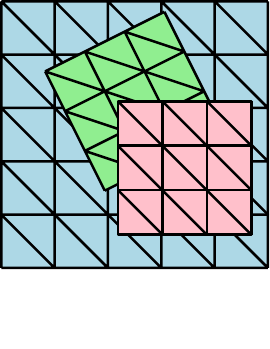}}\qquad\qquad
  \subfloat[]{\label{fig:NO}\includegraphics[height=0.25\linewidth]{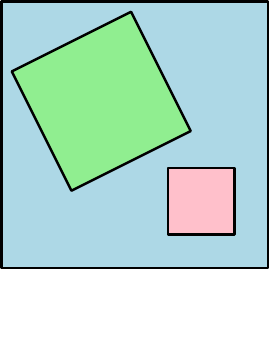}}
  \caption{(a) Given three ordered triangles $K_0$, $K_1$ and $K_2$, the overlaps are $\mcO_{01}$ in green, $\mcO_{02}$ in red and $\mcO_{12}$ in blue. (b) The multimesh of the domains in Figure~\ref{fig:three_domains_b} consists of the active meshes in Figure~\ref{fig:three_meshes_b}. (c) Example with $N=3$ domains and $N_{\OO}=2$ intersecting meshes. (The illustrations (a) and (b) also appear in~\cite{mmfem-1}.)}
\end{figure}

\subsection{Notation for overlaps}
\label{sec:overlaps}
\begin{itemize}[noitemsep,topsep=0pt,parsep=0pt,partopsep=0pt]
\item Let $\OO_i$ denote the \emph{overlap} defined by  $\OO_i = \Omega_{h,i} \setminus \Omega_i$, $i=0,\ldots,N-1$.
\item Let $\OO_{ij} = \OO_i \cap \Omega_j$, $i < j$ be a partition of $\OO_i$. See Figure~\ref{fig:overlap} for an example.
\item For $i < j$, let
  \begin{align}
    \label{eq:indicatorfunction}
    \delta_{ij} =
    \begin{cases}
      1, \quad \OO_{ij} \neq \emptyset,
      \\
      0, \quad \text{otherwise},
    \end{cases}
  \end{align}
be a function indicating which overlaps are non-empty.
For ease of notation, we further let $\delta_{ii} = 1$ for $i = 0, \ldots,N$.
\item Let $N_{\OO} = \max(\max_i \sum_j \delta_{ij}, \max_j \sum_i \delta_{ij})$ be the maximum number of overlaps. Note that $N_\OO$ is bounded by $N$ but is smaller if not all meshes intersect with each other. See Figure~\ref{fig:NO} for an example.
\item Let $N_{\OO_i} = \sum_{j=0}^{i-1} \delta_{ji}$, i.e., the number of meshes below mesh $i$ with non-empty intersection.
\end{itemize}

\subsection{Notation for function spaces}
\begin{itemize}[noitemsep,topsep=0pt,parsep=0pt,partopsep=0pt]

\item Let $W^s_p(\omega)$ denote the standard Sobolev spaces on $\omega\subset \Omega$ with norm denoted by $\| \cdot \|_{W^s_p(\omega)}$ and semi-norm $| \cdot |_{W^s_p(\omega)}$. The special case $p=2$ is denoted by $H^s(\omega)$ and the space with $p=2$ and zero trace is denoted by $H^s_0(\omega)$ (see also e.g.~\cite{BreSco08,Brezis11}). The Euclidean norm on $\R^N$ is denoted by $|\cdot|_{N}$.  The corresponding inner products are labeled accordingly. The same notation is used for the Lebesgue measure and absolute value. It will be clear from the argument which is used.

\item In order to define function spaces on the different meshes in the multimesh configuration we recall the concept of  external direct sums of vector spaces, see for instance \cite{halmos}. Let the external direct sum $X$ of the vector spaces $X_i$ be denoted by $X = \bigoplus_{i=0}^N X_i$, Here $X$ is the vector space with elements  $x \in X$ which are $N+1$ tuples of the form $x =(x_0, \ldots, x_N)$, where $x_i \in X_i$ for $i=0,\dots,N$, equipped with component wise addition
\begin{equation}
x + y = (x_0 + y_0, \dots, x_N + y_N) \qquad x,y \in X,
\end{equation}
and scalar multiplication
\begin{equation}
 t x = (t x_0, \dots t x_N) \qquad t \in \mathbb{R},\ x \in X.
\end{equation}
If the vector spaces $X_i$ are equipped with inner products $(\cdot,\cdot)_{X_i}$  and associated norms $\| \cdot \|_{X_i}$ we let
\begin{equation}
(x,y)_X = \sum_{i=0}^N (x_i, y_i)_{X_i}, \qquad 
\| x \|^2_X = \sum_{i=0}^N \| x_i \|^2_{X_i}.
\end{equation}
Note that the direct sum is a purely algebraic construction, which is not dependent on the fact that the component spaces $X_i$ have specific further properties for instance being function spaces, and since we always have a finite number of component spaces the direct sum is identical to the direct or Cartesian product of spaces, see \cite{halmos} for further details.

\item Let $H^s(\Omega_{h,i})$, $s \geq 0$, be the standard Sobolev spaces of order $s$ on the domain $\Omega_{h,i}$, for $i = 0,\dots, N$, and define the so called multimesh Sobolev space as the external direct sum
\begin{equation}
\bigoplus_{i=0}^N H^s(\Omega_{h,i}).
\end{equation}

\item Let $V_{h,i}$ be a continuous piecewise polynomial finite element space on the partition
$\mcK_{h,i}$ of $\Omega_{h,i}$, and let the multimesh finite element space $V_h$ be the
external direct sum
\begin{equation}
V_h = \bigoplus_{i=0}^N V_{h,i}
\end{equation}
of the finite element spaces $V_{h,i}$, $i=0,\dots,N$. To simplify the presentation we assume homogeneous Dirichlet boundary conditions on $\partial \Omega$ with strong implementation of the boundary conditions in the finite element space, i.e., $v=0$ on $\partial \Omega$ for $v \in V_{h,0}$. Other boundary conditions can be enforced using standard techniques.

\item To represent a function $v \in H^s(\Omega)$, $s\geq 0$, as a multimesh function let $E^\dagger : H^s(\Omega) \hookrightarrow \bigoplus_{i=0}^N H^s(\Omega_{h,i})$
be the embedding defined by
\begin{align}\label{eq:E-dagger}
(E^\dagger v)_i = v|_{\Omega_{h,i}} \qquad i=0,\dots,N.
\end{align}

\item To interpret the multimesh function $v\in \bigoplus_{i=0}^N L^2(\Omega_{h,i}) $ as a function in $L^2(\Omega)$ we define an embedding  $E: \bigoplus_{i=0}^N L^2(\Omega_{h,i}) \hookrightarrow L^2(\Omega)$ as follows
\begin{equation}
(E v_h)|_{\Omega_i} = v_{i} |_{\Omega_i} \qquad i=0,\dots,N.
\end{equation}
We note that it follows from the definitions of $E$ and $E^\dagger$ that
\begin{equation}
E E^\dagger v = v\qquad v \in L^2(\Omega),
\end{equation}
since
\begin{equation}
(E E^\dagger v)|_{\Omega_i} = (E^\dagger v)_i|_{\Omega_i} =  (v|_{\Omega_{h,i}})|_{\Omega_i}
= v|_{\Omega_i},
\end{equation}
where we used the fact $\Omega_i \subset \Omega_{h,i}$ in the last step.
For implementation purposes the following equivalent definition is useful
\begin{equation}\label{eq:E}
(E v)(x) = \max_{\{i \,:\, x \in \Omega_{h,i}\}} v_i(x),
\end{equation}
which corresponds to picking the top-most mesh in the case when $x$ belongs to several meshes.

\item For $s \geq 0$ we define the sum of $V_h$ and $\bigoplus_{i=0}^N H^s(\Omega_{h,i})$ by
\begin{equation}\label{eq:sum-multimesh}
V_h + \bigoplus_{i=0}^N H^s(\Omega_{h,i})  = \bigoplus_{i=0}^N \left( V_{h,i} + H^s(\Omega_{h,i}) \right),
\end{equation}
where each of the component spaces $V_{h,i} + H^s(\Omega_{h,i})\subseteq L^2(\Omega_{h,i})$ consists of functions of the form $v + w$ with $v \in V_{h,i}$ and $w \in H^s(\Omega_{h,i})$.
\end{itemize}

\subsection{Notation for jumps and averages}
\begin{itemize}[noitemsep,topsep=0pt,parsep=0pt,partopsep=0pt]
\item To formulate the stabilization form we define a jump operator for functions
  $v \in V_h$ on the overlaps $\OO_{ik} = \OO_i \cap \Omega_k$, with $i<k$ (cf.~Section \ref{sec:overlaps}), by
\begin{align}
  \llbracket v \rrbracket &= v_i - v_k \qquad \text{on $\OO_{ik}$}.
\end{align}
\item To formulate the Nitsche method we define the  jump and average operators
on the interface segment  $\Gamma_{ij} = \Gamma_i \cap \Omega_j$ for $i>j$ (cf.~Section \ref{sec:interfaces}), as
follows
\begin{align}\label{eq:jump}
   [v] &= v_i - v_j,
   \\ \label{eq:average}
  \langle n_i\cdot \nabla v \rangle &= \kappa_i n_i \cdot \nabla v_{i} + \kappa_j n_{i} \cdot \nabla v_{j},
\end{align}
where $v_i\in V_{h,i}$ and $v_j\in V_{h,j}$. The weights $\kappa_i$ and $\kappa_j$ are defined by
\begin{align}\label{eq:kappa-def}
  \kappa_l = \frac{h_l}{h_i + h_j} \qquad l = i,j,
\end{align}
and we note that
\begin{align}\label{eq:kappa-sum}
\kappa_i + \kappa_j = 1.
\end{align}
\end{itemize}
\begin{remark}
  \label{rem:ij_ji}
By the definitions in Sections \ref{sec:interfaces} and \ref{sec:overlaps} we have $\Gamma_{ij} = \Gamma_i \cap \Omega_j$ and $\OO_{ji} = (\Omega_{h,j} \setminus \Omega_j)\cap \Omega_i \supset \Gamma_{ji}$. Therefore, the two jump operators  $[\cdot]$ on $\Gamma_{ij}$ and
  $\llbracket \cdot \rrbracket$ on $\OO_{ik}$ are compatible on $\Gamma_{ij}$ in the sense that
  \begin{align}
[v] = \llbracket v \rrbracket \qquad \text{on }\Gamma_{ij}\subset {\OO_{ji}}.
\end{align}
The indices are swapped since it is most natural to partition $\Gamma_i$ with respect to domains below $i$, and $\OO_i$ with respect to domains above $i$. We will use this to show coercivity in Proposition~\ref{prop:coercivity} and in the interpolation estimate in Proposition~\ref{prop:interpolation}.
\end{remark}

\subsection{Notation for norms}
\begin{itemize}[noitemsep,topsep=0pt,parsep=0pt,partopsep=0pt]
\item Let $c>0$ and $C>0$ be constants. The inequality $x \leq Cy$ is denoted by $x \lesssim y$. The equivalence $cx\leq y \leq Cx$ is denoted by $x \sim y$.

\item Let $\| \cdot \|_{s_h}$ denote the semi-norm defined by
\begin{align}
  \label{eq:shnorm-def}
  \| v \|_{s_h}^2 =  \sum_{i=0}^{N-1} \sum_{j=i+1}^N \| \llbracket \nabla v \rrbracket \|_{\OO_{ij}}^2,
\end{align}
and let $\| \cdot\|_h$ denote the norm
\begin{align}
  \| v \|_h^2 = \sum_{i=0}^N \| v_i \|_{\Omega_{h,i}}^2.
\end{align}
Note that for the norm $\| \cdot\|_h$, the domain of integration extends to each active domain $\Omega_{h,i}$, meaning that each overlap will be counted (at least) twice.

\item The definition of the energy norm, which we will denote by $\tn \cdot \tn_h$, will be defined after the presentation of the finite element method in~\eqref{eq:energynorm-def}.
\end{itemize}

%---------------------------------------------------------------------------
\section{Finite element method}
\label{sec:formulation}

As a model problem we consider the Poisson problem
\begin{subequations}
  \label{eq:poisson}
  \begin{align}  \label{eq:poisson-a}
    -\Delta u &= f \qquad \text{in } \Omega,\\
    \qquad u &= 0 \qquad \text{on } \partial \Omega,
  \end{align}
\end{subequations}
where $\Omega \subset \R^d$ is a polygonal domain. The multimesh finite element method for~\eqref{eq:poisson} is to find $u_h \in V_h$ such that
\begin{align}
  \label{eq:method}
  A_{h}(u_h, v)  &= l_h(v) \qquad \forall v \in V_h,
\end{align}
where for $v,w \in V_h$,
\begin{align}\label{eq:Ah}
A_h(v,w) & =  a_{h}(v, w) + s_h(v, w),
\\ \nonumber
  a_{h}(v, w) &=
   \sum_{i=0}^N (\nabla v_i,\nabla w_i)_{\Omega_i}
   \\ \nonumber
   &\qquad
   - \sum_{i=1}^N \sum_{j=0}^{i-1} (\langle n_i\cdot \nabla v \rangle, [w])_{\Gamma_{ij}}
      + ([v], \langle n_i \cdot \nabla w \rangle)_{\Gamma_{ij}}
\\ \label{eq:ah}
&\qquad
 + \sum_{i=1}^N \sum_{j=0}^{i-1} \beta_0 ( h_i + h_j )^{-1} ([v], [w])_{\Gamma_{ij}},
  \\ \label{eq:sh}
  s_h (v, w) &= \sum_{i=0}^{N-1} \sum_{j=i+1}^N \beta_1 (\llbracket \nabla v \rrbracket, \llbracket \nabla w \rrbracket)_{\OO_{ij}},
  \\ \label{eq:lh}
  l_h(v) &= \sum_{i=0}^N (f,v_i)_{\Omega_i}.
\end{align}
and we recall that the Dirichlet boundary condition $u=0$ on $\partial \Omega$ 
is for simplicity enforced strongly in $V_{h,0}$. Here, $\beta_0 > 0$ is the Nitsche (interior) penalty parameter and $\beta_1 > 0$ is a stabilization parameter. Note the relation $s_h(v, v) = \beta_1 \|v\|_{s_h}^2$ between the stabilization term $s_h$ and the norm $\|\cdot\|_{s_h}$ \eqref{eq:shnorm-def}.

%---------------------------------------------------------------------------
\section{Energy norm, consistency, Galerkin orthogonality and continuity}
\label{sec:fem_basics}

In the forthcoming analysis we will need to compare the exact solution $u \in H^s(\Omega)$ and the finite element solution $u_h \in V_h$. To facilitate the analysis we will see that it is natural to represent the exact solution $u\in H^s(\Omega)$ as a multimesh function in $\bigoplus_{i=0}^N H^s(\Omega_{h,i})$, using the $E^\dagger$ operator \eqref{eq:E-dagger}, and define the error by
\begin{equation}
E^\dagger u - u_h \in V_h + E^\dagger H^s(\Omega).
\end{equation}
For clarity, when there is no risk of miss understanding we use the simplified
notation $u = E^\dagger u$ and write $u - u_h = E^\dagger u - u_h$.

We define the energy norm $\tn \cdot \tn_h$ on $V_h$ by
\begin{align}
  \tn v \tn^2_h &= \underbrace{\sum_{i=0}^N \| \nabla v_i \|^2_{\Omega_i}}_{I} +  \underbrace{\| v \|^2_{s_h}}_{II}
  +  \underbrace{\sum_{i=1}^N \sum_{j=0}^{i-1}
    ( h_i \| \nabla v_i \|^2_{\Gamma_{ij}} + h_j \| \nabla v_j \|^2_{\Gamma_{ij}} )}_{III}
  \nonumber \\
  &\quad
  +  \underbrace{\sum_{i=1}^N \sum_{j=0}^{i-1} ( h_i + h_j)^{-1} \|[v]\|^2_{\Gamma_{ij}}}_{IV}.
  \label{eq:energynorm-def}
\end{align}
The numbering of the terms will be used to alleviate the analysis of the method.

In order for $A_h$ and the norm $\tn \cdot  \tn_h$ to be well defined also on
$E^\dagger u$ we note that the traces of the normal flux appearing in the forms are
well defined if $u \in H^{3/2+\epsilon}(\Omega)$, $\epsilon > 0$. We then have using
$\Gamma_{ij} \subset \partial \hatOmega_i$ and standard trace inequalities
\begin{equation}\label{eq:H32eps}
\| \nabla v \|_{\Gamma_{ij}} \leq \| \nabla v \|_{\partial \hatOmega_{i}}
\lesssim \| \nabla v \|_{H^{1/2+\epsilon}(\hatOmega_i)}
\lesssim  \| v \|_{H^{3/2+\epsilon}(\Omega)},
\end{equation}
noting that $\hatOmega_i$ does not depend on $h$.
In view of these observations we define
\begin{equation}\label{def:V}
V = E^\dagger H^{3/2 + \epsilon}(\Omega) + V_h.
\end{equation}
We note that $V_h \subset V$, the error $u- u_h = E^\dagger u - u_h\in V$,
 $A_h$ is a bilinear form on  $V$, and $\tn \cdot \tn_h$ is a norm on $V$. We next
  establish the consistency, Galerkin orthogonality and continuity of the form $A_h$.
\begin{prop}[Consistency] \label{prop:consistency}
  The form $A_h$ is consistent; that is,
  \begin{align} \label{eq:consistency}
    A_h(u, v) = l_h(v) \qquad \forall v \in V_h,
  \end{align}
where $u\in H^1_0(\Omega)\cap H^{3/2+\epsilon}(\Omega)$ is the solution to~\eqref{eq:poisson}.
\end{prop}
\begin{proof}
This result follows by for $v = (v_0,\dots,v_N) \in V_h$, multiplying \eqref{eq:poisson-a} by $v_i$, integrating by parts on $\Omega_i \subset \Omega_{h,i}$
and summing the contributions
\begin{align}
&\sum_{i=0}^N (f,v_i)_{\Omega_i} = \sum_{i=0}^N -(\Delta u,v_i)_{\Omega_i}
  =  \sum_{i=0}^N \left( (\nabla u, \nabla v_i)_{\Omega_i} - (n_i \cdot \nabla u,v_i)_{\partial \Omega_i} \right).
\end{align}
For the boundary term we note the decomposition
\begin{align}
\partial \Omega_i = \left( \cup_{j=0}^{i-1} \Gamma_{ij} \right) \cup  \left( \cup_{j=i+1}^{N} \Gamma_{ji} \right)
\end{align}
to conclude that
\begin{align}
\sum_{i=0}^N (n_i \cdot \nabla u,v_i)_{\partial \Omega_i}
=
\sum_{i=0}^N \sum_{j=0}^{i-1} (n_i \cdot \nabla u,v_i)_{\Gamma_{ij}}
+
\sum_{i=0}^N  \sum_{j=i+1}^{N} (n_i \cdot \nabla u,v_i)_{\Gamma_{ji}}.
\end{align}
Using summation by parts and then swapping the indices 
$i$ and $j$, the second term takes the form
\begin{equation}
\sum_{i=0}^N  \sum_{j=i+1}^{N} (n_i \cdot \nabla u,v_i)_{\Gamma_{ji}}  
=
\sum_{j=0}^N  \sum_{i=0}^{j-1} (n_i \cdot \nabla u,v_i)_{\Gamma_{ji}}  
=
\sum_{i=0}^N  \sum_{j=0}^{i-1} (n_j \cdot \nabla u,v_j)_{\Gamma_{ij}}  
\end{equation}
Now on $\Gamma_{ij}$ we obtain using the fact that $n_j = -n_i$ and the definition
 of the jump (\ref{eq:jump}) and average (\ref{eq:average}) with weights summing to one (\ref{eq:kappa-sum}), 
\begin{align}
&(n_i \cdot \nabla u,v_i)_{\Gamma_{ij}}
+ (n_j \cdot \nabla u,v_j)_{\Gamma_{ij}}
\nonumber \\
&\qquad = (n_i \cdot \nabla u,v_i)_{\Gamma_{ij}}
- (n_i \cdot \nabla u,v_j)_{\Gamma_{ij}}
\\
&\qquad = (n_i \cdot \nabla u, v_i - v_j)_{\Gamma_{ij}}
\\
&\qquad =
(\langle n_i \cdot \nabla u \rangle, [ v ])_{\Gamma_{ij}}.
\end{align}
Here all traces of the gradient are well defined in view of \eqref{eq:H32eps}. Now observing
that  $u|_{\Omega_i} = (E^\dagger u)_i |_{\Omega_i}$ we obtain using the notation
$u_i = (E^\dagger u)_i$,
\begin{align}
\sum_{i=0}^N (f,v_i)_{\Omega_i}
=
\sum_{i=0}^N (\nabla u_i, \nabla v_i)_{\Omega_i} - \sum_{i=0}^N \sum_{j=0}^{j-1}
(\langle n_i \cdot \nabla u_i \rangle, [ v ])_{\Gamma_{ij}},
\end{align}
which combined with the observation that
\begin{equation}
[u ] = [ E^\dagger u ] = u_i - u_j = 0 \qquad \text{on $\Gamma_{ij}$},
\end{equation}
and similarly
\begin{equation}
\llbracket u \rrbracket = \llbracket E^\dagger u \rrbracket  = 0,  
\qquad \text{on $\mcO_{ij}$}.
\end{equation}
\end{proof}

\begin{prop}[Galerkin orthogonality] \label{prop:galerkin}
  The form $A_h$ satisfies the Galerkin orthogonality; that is,
  \begin{align} \label{eq:galerkin}
    A_h(u - u_h, v) = 0 \qquad \forall v \in V_h,
  \end{align}
  where $u\in H^1_0(\Omega)\cap H^{3/2+\epsilon}(\Omega)$ is the solution of~\eqref{eq:poisson} and $u_h\in V_h$ is a solution of~\eqref{eq:method}.
\end{prop}
\begin{proof}
  The result follows directly from Proposition~\ref{prop:consistency} and \eqref{eq:method}.
\end{proof}

\begin{prop}[Continuity] \label{prop:continuity}
  The form $A_h$ is continuous; that is,
  \begin{align}
    \label{eq:cont}
    A_h(v,w) \lesssim \tn v \tn_h \tn w \tn_h \qquad \forall v,w \in V,
  \end{align}
  where $V$ is defined in \eqref{def:V}.
\end{prop}
\begin{proof} The result follows by repeated use of the Cauchy--Schwarz inequality.
\end{proof}

%---------------------------------------------------------------------------
\section{Coercivity}
\label{sec:coercive}

To prove that the form $A_h$ is coercive, we will make use of the following lemma.
\begin{lemma}
  \label{lem:A}
  For all $v\in V_h$, we have
  \begin{align}
    \label{eq:lem:A}
    \|\nabla v \|_h^2
    \lesssim
    N_\OO \sum_{i=0}^N \| \nabla v_i \|^2_{\Omega_i} +  \| v \|_{s_h}^2.
  \end{align}
\end{lemma}
\begin{proof}
Take an element $K \in\mcK_{h,i}$ and observe that $K = \cup_{j=i}^N K \cap \Omega_j$.
It follows that
\begin{align}
\| \nabla v_i \|^2_{K}
&=\| \nabla v_i \|^2_{K \cap \Omega_i}
+
\sum_{j=i+1}^N \| \nabla v_i \|^2_{K\cap\Omega_j}
\\
&\leq
\| \nabla v_i \|^2_{K \cap \Omega_i}
+
2\sum_{j=i+1}^N ( \| \nabla (v_i - v_j)  \|^2_{K\cap\Omega_j} + \| \nabla v_j \|^2_{K\cap\Omega_j} )
\\
&\leq
2 \sum_{j=i}^N \| \nabla v_j \|^2_{K \cap \Omega_j}
+
2\sum_{j=i+1}^N \| \nabla (v_i - v_j)  \|^2_{K\cap\Omega_{j}}.
\end{align}
Here we have made use of the inequality $a^2 \leq 2(a-b)^2 + 2b^2$, which follows by Young's inequality $2ab \leq a^2 + b^2$ applied to $a^2 = (a - b + b)^2 = (a - b)^2 + b^2 + 2(a-b)b$.

Summing over all elements $K \in \mcK_{h,i}$ we have
\begin{align}
\| \nabla v_i \|^2_{\Omega_{h,i}}
&\leq
2 \sum_{j=i}^N  \| \nabla v_j \|^2_{\Omega_{h,i} \cap \Omega_j}
+
2\sum_{j=i+1}^N \| \nabla (v_i - v_j)  \|^2_{\Omega_{h,i} \cap\Omega_{j}} \\
&=
2 \sum_{j=i}^N \delta_{ij} \| \nabla v_j \|^2_{\Omega_{h,i} \cap \Omega_j}
+
2\sum_{j=i+1}^N \| \nabla (v_i - v_j)  \|^2_{\OO_{ij}} \\
&\leq
2 \sum_{j=i}^N \delta_{ij} \| \nabla v_j \|^2_{\Omega_j}
+
2\sum_{j=i+1}^N \| \nabla (v_i - v_j)  \|^2_{\OO_{ij}},
\end{align}
where we have used $\Omega_{h,i} \cap \Omega_j = \OO_{ij} \subseteq \Omega_j$ for $i < j$. Note that the second sum is empty for $i = N$.

Summing over all domains, we obtain by~\eqref{eq:indicatorfunction}
\begin{align}
  \| \nabla v \|_h^2
&\leq
2\sum_{i=0}^N  \sum_{j=i}^N \delta_{ij} \| \nabla v_j \|^2_{\Omega_j}
+
2 \sum_{i=0}^{N-1} \sum_{j=i+1}^N \| \nabla (v_i - v_j)  \|^2_{\OO_{ij}} \\
&\leq
2N_\OO \sum_{j=0}^N  \| \nabla v_j \|^2_{\Omega_j}
+
2 \|v\|_{s_h}^2,
\end{align}
which proves the estimate.
\end{proof}

\begin{remark}\label{rem:number-meshes} There is a dependence on the maximum number of overlaps $N_\OO$ in Lemma~\ref{lem:A}. In practice, $N_\OO$ is of moderate size and this dependence is not an issue. The interpolation error estimates and condition number estimates (shown below) have a different kind of dependence.
\end{remark}

\begin{remark}
Using an inverse bound of the form (see e.g.~\cite{BreSco08})
\begin{align}
  \label{eq:general_inverse}
  \| v \|_{H^l(K)} \lesssim h^{m-l} | v |_{H^m(K)} \qquad m,l\in \mathbb{Z}^+, \ m \leq l,
\end{align}
one can show that the stabilization term $s_h(v, w)$ may alternatively be formulated as
\begin{align}\label{eq:sh-L2}
  s_h (v, w) = \sum_{i=0}^{N-1} \sum_{j=i+1}^N \beta_1 (h_i + h_j)^{-2} (\llbracket v \rrbracket, \llbracket w \rrbracket)_{\OO_{ij}}.
 \end{align}
\end{remark}

Using Lemma~\ref{lem:A}, we may now proceed to prove the coercivity of the bilinear form.
\begin{prop}[Coercivity] \label{prop:coercivity}
The form $A_h$ is coercive. More precisely, for $\beta_0$ and $\beta_1$ large enough, we have
\begin{align}
  \label{eq:coercive}
\tn v \tn^2_h  \lesssim A_h(v,v)\qquad \forall v \in V_h.
\end{align}
\end{prop}
\begin{proof}
  We first note that
  \begin{align}\label{eq:coer-a}
   A_h(v,v) &\geq
   \sum_{i=0}^N \|\nabla v_i \|^2_{\Omega_i}
   + \sum_{i=1}^N \sum_{j=0}^{i-1} \beta_0 ( h_i + h_j)^{-1} \|[v]\|^2_{\Gamma_{ij}}
   + \beta_1 s_h(v, v)
   \nonumber \\
   &\qquad
   - \underbrace{\sum_{i=1}^N \sum_{j=0}^{i-1} 2|(\langle n_i\cdot \nabla v \rangle, [v])_{\Gamma_{ij}} |}_{\bigstar}.
\end{align}
Now, for $l = i$ or $l = j$, let
\begin{align}
  \label{eq:cutting}
  \mcK_{h,l}(\Gamma_{ij}) = \{ K \in \mcK_{h,l} : K \cap \Gamma_{ij} \neq \emptyset \}
\end{align}
denote the set of elements in $\mcK_{h,l}$ which intersect $\Gamma_{ij}$. Using an inverse estimate (see~\cite{Hansbo:2003aa}), we have
\begin{align}
  \label{eq:inverse}
  h_l \| \nabla v_l \|^2_{K \cap \Gamma_{ij}} \lesssim \| \nabla v_l \|^2_K,
\end{align}
where the constant is independent of the position of $\Gamma_{ij}$. It follows that
\begin{align}
  \label{eq:inverse_i}
  \sum_{j=0}^{i-1} h_l \| \nabla v_l \|^2_{\Gamma_{ij}}
  &=
  \sum_{j=0}^{i-1} \delta_{ji} h_l \| \nabla v_l \|^2_{\Gamma_{ij}}
  \\
  &\lesssim
  \sum_{j=0}^{i-1} \delta_{ji} \| \nabla v_l \|^2_{\mcK_{h,l}(\Gamma_{ij})}
  \\
  &\leq
  \sum_{j=0}^{i-1} \delta_{ji} \| \nabla v_l \|^2_{\Omega_{h,l}},
  \end{align}
where we have noted that $\Gamma_{ij}$ (the part of $\Gamma_i$ bordering to $\Omega_j$ for $j < i$) is empty if the overlap $\mathcal{O}_{ji}$ (the part of $\Omega_{h,j}$ intersected by $\Gamma_i$ for $j < i$) is empty, as indicated by $\delta_{ji}$. See also Remark \ref{rem:ij_ji}.

For $l = i$, we thus obtain the estimate
\begin{align}
  \sum_{i=1}^{N} \sum_{j=0}^{i-1} h_i \| \nabla v_i \|^2_{\Gamma_{ij}}
  &\lesssim
  \sum_{i=1}^{N} \sum_{j=0}^{i-1} \delta_{ji} \| \nabla v_i \|^2_{\Omega_{h,i}} \\
  &\leq
  \sum_{i=1}^{N} \left(\sum_{j=0}^{N} \delta_{ji}\right) \| \nabla v_i \|^2_{\Omega_{h,i}} \\
  &=
  \sum_{i=1}^{N} N_{\OO_i} \| \nabla v_i \|^2_{\Omega_{h,i}} \\
  &\leq N_\OO \sum_{i=0}^{N} \| \nabla v_i \|^2_{\Omega_{h,i}} \\
  &=
  \label{eq:trace_i}
  N_\OO \| \nabla v \|_h^2,
\end{align}
while for $l = j$, we obtain the similar estimate
\begin{align}
  \sum_{i=1}^{N} \sum_{j=0}^{i-1} h_j \| \nabla v_j \|^2_{\Gamma_{ij}}
  &\lesssim
  \sum_{i=1}^{N} \sum_{j=0}^{i-1} \delta_{ji} \| \nabla v_j \|^2_{\Omega_{h,j}} \\
  &\leq
  \sum_{j=0}^{N} \left(\sum_{i=0}^{N} \delta_{ji}\right) \| \nabla v_j \|^2_{\Omega_{h,j}} \\
  &=
  N_\OO \sum_{j=0}^{N} \| \nabla v_j \|^2_{\Omega_{h,j}} \\
  &=
  \label{eq:trace_j}
  N_\OO \| \nabla v \|_h^2.
\end{align}

Proceeding with $\bigstar$ using the Cauchy--Schwarz inequality with weight $\epsilon (h_i + h_j)$, where $\epsilon$ is a positive number, and Young's inequality $2ab \leq a^2 + b^2$  we obtain
\begin{align}
&\bigstar
=
\sum_{i=1}^N \sum_{j=0}^{i-1} 2|(\langle n_i\cdot \nabla v \rangle, [v])_{\Gamma_{ij}} | \\
&\leq
\sum_{i=1}^N\sum_{j=0}^{i-1} 2 \epsilon^{1/2} (h_i+h_j)^{1/2}\| \langle n_i\cdot \nabla v \rangle\|_{\Gamma_{ij}}  \epsilon^{-1/2} (h_i + h_j)^{-1/2}\|[v]\|_{\Gamma_{ij}}
\\
&\leq
\sum_{i=1}^N\sum_{j=0}^{i-1} \epsilon (h_i + h_j) \|  \langle n_i\cdot \nabla v \rangle\|_{\Gamma_{ij}}^2
+ \sum_{i=1}^N\sum_{j=0}^{i-1} \epsilon^{-1} (h_i + h_j)^{-1} \|[v]\|_{\Gamma_{ij}}^2
\\ \label{eq:coer-bb}
&\lesssim
\sum_{i=1}^N\sum_{j=0}^{i-1} \epsilon \big(  h_i \| \nabla v_i \|^2_{\Gamma_{ij}}
+ h_j \| \nabla v_j \|^2_{\Gamma_{ij}} \big)
+ \sum_{i=1}^N\sum_{j=0}^{i-1} \epsilon^{-1} ( h_i + h_j )^{-1} \|[v]\|^2_{\Gamma_{ij}}
\\
\label{eq:coer-c}
&\lesssim
\epsilon N_\OO \| \nabla v \|^2_{h}
+ \sum_{i=1}^N\sum_{j=0}^{i-1} \epsilon^{-1} ( h_i + h_j )^{-1}   \|[v]\|^2_{\Gamma_{ij}}.
\end{align}
In~\eqref{eq:coer-bb} we used that
\begin{align}
  &(h_i + h_j) \|  \langle n_i\cdot \nabla v \rangle\|_{\Gamma_{ij}}^2
  \\
  &\qquad \leq
  2(h_i + h_j)\kappa^2_i   \|   \nabla v_i \|_{\Gamma_{ij}}^2
  +2(h_i + h_j)\kappa^2_j    \|   \nabla v_j \|_{\Gamma_{ij}}^2
  \\
  &\qquad \leq
  2 h_i  \|   \nabla v_i \|_{\Gamma_{ij}}^2
  +2 h_j    \|   \nabla v_j \|_{\Gamma_{ij}}^2.
\end{align}
where we used the definition \eqref{eq:kappa-def} of the weights $\kappa_l$ to obtain
\begin{align}
  (h_i + h_j) \kappa_l^2 = \frac{h_l^2}{h_i + h_j} =  \frac{h_l}{h_i + h_j} h_l \leq h_l \qquad l = i,j.
\end{align}
In~\eqref{eq:coer-c} we made use of \eqref{eq:trace_i} and \eqref{eq:trace_j}.

By Lemma~\ref{lem:A}, we may now estimate the $\|\cdot\|_h$ norm in terms of the $\|\cdot\|_\Omega$ norm and the $\|\cdot\|_{s_h}$ norm to obtain
\begin{align}
\bigstar
&\lesssim
\label{eq:coer-d}
\epsilon N_\OO^2 \sum_{i=0}^N \|\nabla v_i\|_{\Omega_i}^2
+ \epsilon N_\OO \|v\|_{s_h}^2
+ \epsilon^{-1} \sum_{i=1}^N \sum_{j=0}^{i-1} (h_i + h_j)^{-1} \|[v]\|^2_{\Gamma_{ij}}.
\end{align}

Combining~\eqref{eq:coer-a} and~\eqref{eq:coer-d}, we find that
\begin{align}
  \label{eq:Ah_final}
A_h(v,v) &\geq
\sum_{i=0}^N (1 - \epsilon C N_\OO^2) \|\nabla v_i \|^2_{\Omega_i}
+ (\beta_1 - \epsilon CN_\OO) \| v \|_{s_h}^2
 \nonumber \\
 &\qquad + \sum_{i=1}^N \sum_{j=0}^{i-1} ( \beta_0 - \epsilon^{-1}C) (h_i + h_j)^{-1} \|[v]\|^2_{\Gamma_{ij}},
\end{align}
and thus, by choosing $\epsilon$ small enough and then $\beta_0$ and $\beta_1$ large enough,
\begin{align}
A_h(v,v) &\gtrsim
\sum_{i=0}^N  \|\nabla v_i \|^2_{\Omega_i}
+ \| v \|_{s_h}^2 + \sum_{i=1}^N \sum_{j=0}^{i-1} (h_i + h_j)^{-1} \|[v]\|^2_{\Gamma_{ij}}.
 \end{align}

The coercivity now follows by noting that term $III$ in~\eqref{eq:energynorm-def} may be controlled by terms $I$ and $II$ as above
in the estimate of $\bigstar$.
\end{proof}

\begin{remark}
  By continuity~\eqref{eq:cont}, coercivity~\eqref{eq:coercive} and a continuous $l_h(v)$, there exists a unique solution to~\eqref{eq:method} by the Lax--Milgram theorem (see e.g.~\cite{BreSco08}).
\end{remark}

\begin{remark} In view of (\ref{eq:Ah_final}) we note that for large $\beta_0$ we may take
$\epsilon \sim \beta_0^{-1}$ and thus we can choose $\beta_1 \sim \epsilon \sim \beta_0^{-1}$.
\end{remark}

%---------------------------------------------------------------------------
\section{Interpolation error estimate}
\label{sec:interpolationerrest}

To construct an interpolation operator into $V_h$, we pick a standard interpolation operator into $V_{h,i}$,
\begin{align}
  \pi_{h,i}: H^1(\Omega_{h,i}) \longrightarrow V_{h,i}, \qquad i=0,1,\dots,N,
\end{align}
where $\pi_{h,i}$ satisfies the standard interpolation error estimate (see e.g.~\cite{BreSco08})
\begin{align}
  \label{eq:interpol-element}
  \| v - \pi_{h,i} v \|_{H^m(K)} &\lesssim h^{k+1-m} | v |_{H^{k+1}(\mcN_h(K))}.
\end{align}
Here, $\mcN_h(K)$ denotes the set of elements that share a vertex with $K$. We then define
the interpolation operator
\begin{equation}
  \displaystyle
\Pi_h : \bigoplus_{i=0}^N H^1(\Omega_{h,i}) \ni v \longmapsto \oplus_{i=0}^N \pi_{h,i} v_i \in V_h,
\end{equation}
which by composition with $E^\dagger$, see \eqref{eq:E-dagger}, provides the interpolation operator
\begin{equation}
  \displaystyle
\pi_h : H^1(\Omega)  \ni v \longmapsto \Pi_h E^\dagger v =  \oplus_{i=0}^N \pi_{h,i} (v|_{\Omega_{h,i}}) \in V_h.
\end{equation}

\begin{figure}
  \centering
  \includegraphics[width=0.3\linewidth]{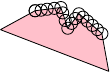}
  \caption{Balls $B_\delta(x)$ centered at $x\in \Gamma$.}
  \label{fig:balls}
\end{figure}

To prove an interpolation error estimate for $\pi_h$, recall Remark~\ref{rem:boundary-overlap} regarding the non-intersecting boundaries and let $U_\delta(\Gamma_{ij})$ denote the tubular neighborhood of $\Gamma_{ij}$ defined by
\begin{align}
  \label{eq:Uloc}
  U_\delta(\Gamma_{ij}) = \bigcup_{x\in \Gamma_{ij}} B_\delta(x),
\end{align}
where $B_\delta(x)$ is the ball of radius $\delta$ centered at $x$; see Figure \ref{fig:balls}. In addition, let
\begin{align}
  \label{eq:Uglob}
  U_\delta(\Gamma) = \bigcup_{i,j} U_\delta(\Gamma_{ij}).
\end{align}

\begin{prop}[Interpolation error estimate] \label{prop:interpolation}
  The interpolation operator $\pi_h$ satisfies the interpolation error estimate
  \begin{align}\label{eq:interpol-energy}
    \tn v - \pi_h v \tn^2_h \lesssim C_{h,N}
    \sum_{i=0}^N h_i^{2k} | v_i |^2_{H^{k+1}(\Omega_{h,i}) \cap W^{k+1}_\infty(U_h(\Gamma_i))},
  \end{align}
  where
  \begin{align}\label{eq:interpol-energy-constant}
  C_{h,N} =  1 + \max_{0\leq i \leq N} h^d_i N_{\OO_i}
  + \max_{0\leq i \leq N} h_i |\Gamma_i|,
  \end{align}
 and the norm is defined by
\begin{align}
 | v |^2_{H^{k+1}(\Omega) \cap W^{k+1}_\infty(U_h(\Gamma))} =| v |^2_{H^{k+1}(\Omega)} + | v |^2_{W^{k+1}_\infty(U_h(\Gamma))}.
\end{align}
\end{prop}
\begin{proof} We first let $\eta =  v - \pi_h v$ denote the interpolation error and recall
 the numbering of the terms in the definition of the energy norm  $\tn \cdot \tn_h$~\eqref{eq:energynorm-def}. Starting with term $I$, we have
\begin{align}
  \label{eq:term1}
  I(\eta)
  &=
  \sum_{i=0}^N  \| \nabla \eta_i \|_{\Omega_i}^2
  \leq
  \sum_{i=0}^N  \| \nabla \eta_i \|_{\Omega_{h,i}}^2.
\end{align}

For term $II$, we have, since $\cup_{j=i+1}^N \OO_{ij} \subseteq \Omega_{h,i}$ and $\OO_{ij} \subseteq U_\delta(\Gamma_{ji}) \cap \Omega_j$ (see Remark \ref{rem:ij_ji}) with $\delta \sim h_i$,
\begin{align}
  II(\eta)
  &\lesssim
  \sum_{i=0}^{N-1} \sum_{j=i+1}^N
  (\|\nabla \eta_i \|^2_{\OO_{ij}}
  + \|\nabla \eta_j \|^2_{\OO_{ij}})
  \\
  \label{eq:term2}
  &\leq
  \sum_{i=0}^{N-1} \| \nabla \eta_i \|_{\Omega_{h,i}}^2
  + \sum_{i=0}^{N-1} \sum_{j=i+1}^N \delta_{ij} \| \nabla \eta_j \|_{U_\delta(\Gamma_{ji}) \cap \Omega_j}^2.
\end{align}

For term $III$, recall the inverse estimate \eqref{eq:inverse_i} and note that $\mcK_{h,j(\Gamma_{ij})} \subseteq U_\delta(\Gamma_{ij})$ with $\delta \sim h_j$. Thus,
\begin{align}
III(\eta)
&=
\sum_{i=1}^N \sum_{j=0}^{i-1}
( h_i \| \nabla \eta_i \|_{\Gamma_{ij}}^2
+ h_j \| \nabla \eta_j \|_{\Gamma_{ij}}^2 )
\\
\label{eq:term3}
&\lesssim
N_\OO \sum_{i=1}^N \| \nabla \eta_i \|_{\Omega_{h,i}}^2
+ \sum_{i=1}^N \sum_{j=0}^{i-1} \delta_{ji} \| \nabla \eta_j \|_{U_\delta(\Gamma_{ij})}^2.
\end{align}

For term $IV$, we first handle the jump term as in $II$ and then proceed as for $III$,
\begin{align}
  IV(\eta)
  &\leq
  \sum_{i=1}^N \sum_{j=0}^{i-1} ( h_i + h_j )^{-1} ( \| \eta_i \|^2_{\Gamma_{ij}} + \| \eta_j \|^2_{\Gamma_{ij}} )
  \\
  \label{eq:term4}
  &\lesssim
  N_\OO \sum_{i=1}^N h_i^{-2} \| \eta_i \|_{\Omega_{h,i}}^2
  +
  \sum_{i=1}^N \sum_{j=0}^{i-1} \delta_{ji} h_j^{-2} \| \eta_j \|_{U_\delta(\Gamma_{ij})}^2.
\end{align}

Now, note that since $\Omega_{h,i} \subseteq \Omega_i \cup_{j=i+1}^N U_\delta(\Gamma_{ij})$, we have
\begin{align}
  \sum_{i=0}^N \| v_i \|_{\Omega_{h,i}}^2
  &\lesssim
  \sum_{i=0}^N \| v_i \|_{\Omega_i}^2
  + \sum_{i=1}^N \sum_{j=0}^{i-1} \delta_{ij} \| v_i \|_{U_\delta(\Gamma_{ij})}^2.
\end{align}
Therefore, there are only two terms in $I$ -- $IV$ that need to be estimated. First
\begin{align}
  h_i^{2(m-1)} | \eta_i |_{H^m(\Omega_i)}^2
  \lesssim
  h_i^{2k} | v_i |_{H^{k+1}(\Omega_i)}^2 \qquad m=0,1,
\end{align}
which follows immediately by~\eqref{eq:interpol-element}. Second, we make use of the disjoint partition of $\Gamma_i$ and noting that
\begin{align}
|U_\delta(\Gamma_{ij})|
\lesssim
h_i \max( h_i^{d-1}, |\Gamma_{i}|),
\end{align}
we obtain
\begin{align}
 &   \sum_{i=1}^N \sum_{j=0}^{i-1}  h_i^{2(m-1)} | \eta_i |_{H^m(U_\delta(\Gamma_{ij}))}^2
  \nonumber \\
   &\qquad \lesssim
   \sum_{i=1}^N \sum_{j=0}^{i-1}  h_i^{2k} | v_i |_{H^{k+1}(U_\delta(\Gamma_{ij}))}^2
    \\
    &\qquad \lesssim
    \sum_{i=1}^N \sum_{j=0}^{i-1}  h_i^{2k} h_i \max (h_i^{d-1},|\Gamma_{ij}|) | v_i |_{W^{k+1}_\infty(U_\delta(\Gamma_{ij}))}^2
    \\
   &\qquad \lesssim
    \sum_{i=1}^N \left(  h_i \sum_{j=0}^{i-1}  \max(h_i^{d-1},|\Gamma_{ij}|) \right)
    h_i^{2k}  | v_i |_{W^{k+1}_\infty(U_\delta(\Gamma_{i}))}^2
    \\
   &\qquad \lesssim
    \sum_{i=1}^N (h^d_i N_{\OO_i} + h_i |\Gamma_i|)
    h_i^{2k}  | v_i |_{W^{k+1}_\infty(U_\delta(\Gamma_{i}))}^2.
\end{align}
Due to the maximum norm, this estimate also holds with $\eta_i$ replaced by $\eta_j$, and the desired estimate holds.
\end{proof}
\begin{remark} Note that the stronger control  $v_i \in W^{k+1}_\infty(U_h(\Gamma))$ enables
us to establish the estimate \eqref{eq:interpol-energy} with the constant given by \eqref{eq:interpol-energy-constant} which only have weak dependence on the configuration of the overlapping mesh configuration encoded in terms of $N_{\mcO_i}$ and $|\Gamma_i|$.
\end{remark}

%---------------------------------------------------------------------------
\section{\emph{A~priori} error estimates}
\label{sec:apriori}

We may now prove the following optimal order \emph{a~priori} error estimates.
The estimates are supported by the numerical results presented in Figure~\ref{fig:poisson_convergence}. For details on these results, we refer to the accompanying paper~\cite{mmfem-1}.

\begin{theorem}[\emph{A~priori} error estimates] \label{th:apriori}
  The finite element solution $u_h$ of~\eqref{eq:method} satisfies the following \emph{a~priori} error estimates:
  \begin{align}\label{eq:energy}
    \tn u - u_h \tn_h^2 &\lesssim C_{h,N}  \sum_{i=0}^N h_i^{2k} | u |^2_{H^{k+1}(\Omega_{h,i}) \cap W^{k+1}_\infty(U_h(\Gamma_i))},\\
    \label{eq:L2}
    \| u - u _h \|_\Omega^2 &\lesssim  ( N_{\mcO} + 1)^{1/2} C_{h,N} \hmax \sum_{i=0}^N h_i^{2k} | u |^2_{H^{k+1}(\Omega_{h,i}) \cap W^{k+1}_\infty(U_h(\Gamma_i))} .
  \end{align}
\end{theorem}
\begin{proof}
  The proof of~\eqref{eq:energy} follows the standard procedure of splitting the error and using the energy norm interpolation error estimate from Proposition~\ref{prop:interpolation},
\begin{align}
  \tn u - u_h \tn_h
  &\leq
  \label{eq:proof-energy-a}
  \tn u - \pi_h u \tn_h
  + \tn \pi_h u - u_h \tn_h.
\end{align}
To estimate the second term on the right-hand side of~\eqref{eq:proof-energy-a},
we use the coercivity (Proposition~\ref{prop:coercivity}), Galerkin orthogonality (Proposition~\ref{prop:galerkin}) and continuity (Proposition~\ref{prop:continuity}) of $A_h$ to obtain
\begin{align}
\tn \pi_h u - u_h \tn_h^2
&\lesssim
A_h(  \pi_h u - u_h,  \pi_h u - u_h )
\\
&=
A_h(  \pi_h u - u,  \pi_h u - u_h )
\\
&\lesssim
\tn  \pi_h u - u \tn_h \tn \pi_h u - u_h \tn_h.
\end{align}
It follows that
\begin{align}
  \label{eq:proof-energy-b}
  \tn \pi_h u - u_h \tn_h &\lesssim \tn u - \pi_h u \tn_h.
\end{align}
Combining~\eqref{eq:proof-energy-a} and~\eqref{eq:proof-energy-b} with the interpolation error estimate of Proposition~\ref{prop:interpolation} now yields~\eqref{eq:energy}.

To prove~\eqref{eq:L2}, we use a standard duality argument (see e.g.~\cite{BreSco08}). Let $\phi$ be the solution to the dual problem
\begin{subequations}
  \begin{align}
    -\Delta \phi &= \psi \qquad \text{in } \Omega,\\
    \phi &= 0 \qquad  \text{on } \Omega,
  \end{align}
\end{subequations}
with $\psi \in L^2(\Omega)$. Using elliptic regularity we then have
\begin{equation}
\| \phi \|_{H^2(\Omega)} \lesssim \| \psi \|_\Omega,
\end{equation}
from which it follows that $\phi\in V$. Using the fact that $A_h$ is symmetric it follows from consistency (Proposition \ref{prop:consistency}) that
\begin{align}
  A_h(v,\phi) = ( v, \psi)_\Omega \qquad \forall v \in V.
\end{align}
We now take $v= e = u - u_h$ and use the Galerkin orthogonality (Proposition~\ref{prop:galerkin}), continuity (Proposition~\ref{prop:continuity}) and a standard
interpolation inequality on each set $\Omega_{h,i}$ (note that we cannot use stronger
regularity than $\phi \in H^2(\Omega)$ since $\psi \in L^2(\Omega)$ and thus the interpolation bound (\ref{prop:interpolation}) is not applicable for $\phi$) to obtain
\begin{align}
(e,\psi)_\Omega
&=A_h(e,\phi)
\\
&=A_h(e,\phi - \pi_h \phi )
\\
&\leq
\label{eq:continuity-regularity}
\tn e \tn_h \tn \phi - \pi_h \phi \tn_h
\\
&\lesssim \tn e \tn_h \left( \sum_{i=0}^N h_i^{2} | \phi |^2_{H^{2}(\Omega_{h,i})} \right)^{1/2}
\\
&\lesssim \tn e \tn_h \left( \sum_{i=0}^N h_i^{2} | \phi |^2_{H^{2}(\Omega_{h,i}\setminus \Omega_i)} + \sum_{i=0}^N h_i^{2} | \phi |^2_{H^{2}(\Omega_i)}\right)^{1/2}
\\
&\lesssim \tn e \tn_h \left( (N_{\mcO} + 1) \hmax^2  | \phi |^2_{H^{2}(\Omega)} \right)^{1/2}
\\ \label{eq:proof-energy-c}
&\lesssim \tn e \tn_h (N_{\mcO} + 1)^{1/2} \hmax \| \psi \|_{\Omega},
\end{align}
where we used the fact that the maximum number of overlapping meshes is $N_{\mcO}$
 and in the last step we have used the standard elliptic regularity estimate (see e.g.~\cite{BreSco08}). Note that we have continuity~\eqref{eq:continuity-regularity} also for functions in $H^{3/2+\epsilon}(\Omega)$, $\epsilon>0$, as noted in Proposition~\ref{prop:continuity}. The desired estimate~\eqref{eq:L2} now follows from~\eqref{eq:proof-energy-c} by~\eqref{eq:energy} and taking $\psi = e$.
\end{proof}

\begin{figure}
  \begin{center}
    \includegraphics[width=0.45\textwidth]{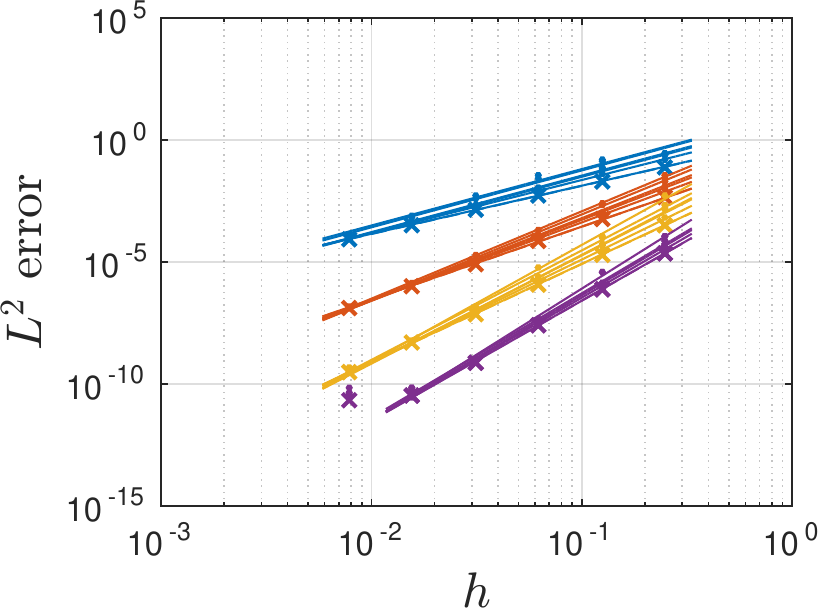}
    \includegraphics[width=0.45\textwidth]{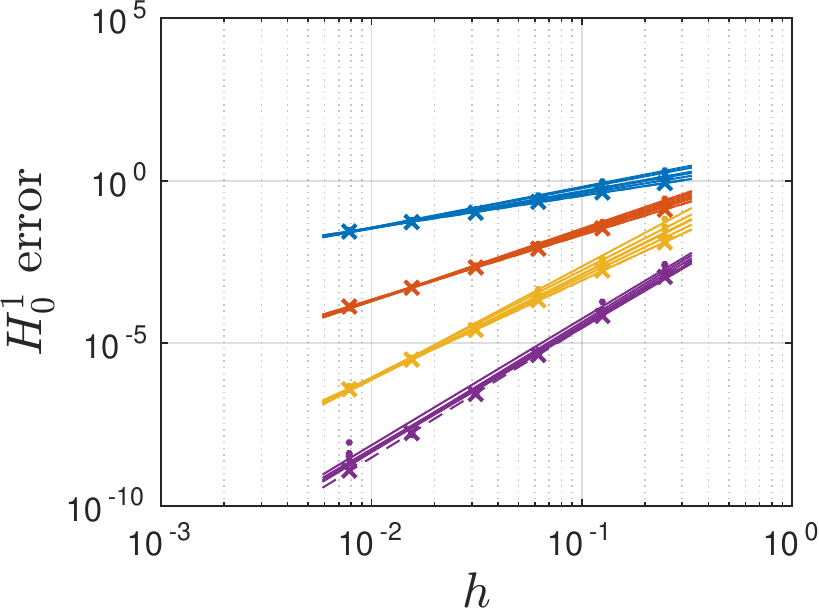}%
  \end{center}
  \caption{Rate of convergence in the $L^2(\Omega)$ (left) and $H^1_0(\Omega)$ (right) norms for $p=1$ (blue), $p=2$ (red), $p=3$
    (yellow) and $p=4$ (purple), where $p$ is the polynomial degree of the finite element approximation. For each $p$, the convergence rate is shown for $N = 1, 2, 4, 6, 16, 32$ meshes (six lines) and the errors for $N = 0$ (the standard single mesh discretization) are marked with $\times$ and dashed lines. (These illustrations also appear in~\cite{mmfem-1}.)}
  \label{fig:poisson_convergence}
\end{figure}

%---------------------------------------------------------------------------
\section{Condition number estimate}
\label{sec:conditioning}

To prove a bound on the condition number, we first introduce some notation and definitions. Let $\{\varphi_{i,j}\}_{j=1}^{M_i}$ be the finite element basis of $V_{h,i}$. We then have the expansion
\begin{align}
  v_i = \sum_{j=1}^{M_i} \hatv_{i,j} \varphi_{i,j},
\end{align}
for each part $v_i$ of a multimesh function $v = (v_0,\ldots,v_N)$. Collecting all expansion coefficients for the $1 + N$ parts into a vector $\hat{v}$ of dimension $M = \sum_{i=0}^N M_i$, the total stiffness matrix $\hatA$ for the multimesh system is defined by
\begin{align}
  \label{eq:def-stiffness-matrix}
  (\hatA \hatv,\hatw)_{M} = A_h(v,w) \qquad \forall v,w \in V_h,
\end{align}
with condition number
\begin{align}
  \label{eq:def-condition-number}
  \kappa(\hatA) = | \hatA |_{M} | \hatA^{-1} |_{M}.
\end{align}

To derive an estimate of $\kappa(\hatA)$ we make use of the following Lemmas.
\begin{lemma}[Inverse inequality]
  \label{lem:inverse-cond}
  It holds that
  \begin{align}
    \label{eq:inverse-cond}
    \tn v \tn_h^2 \lesssim (1 + N_\OO) \max_{0\leq i \leq N} h_i^{-2} \| v \|_h^2 \qquad \forall v \in V_h.
  \end{align}
\end{lemma}
\begin{proof}
  Recall the definition of the energy norm~\eqref{eq:energynorm-def}.
  We first note that
\begin{align}
  I
  =
  \sum_{i=0}^N \| \nabla v_i \|_{\Omega_i}^2
  \leq
  \sum_{i=0}^N \| \nabla v_i \|_{\Omega_{h,i}}^2.
\end{align}

For term $II$, we have by recalling~\eqref{eq:term2}
\begin{align}
  II
  &\lesssim
  \sum_{i=0}^{N-1} \| \nabla v_i \|_{\Omega_{h,i}}^2
  +
  \sum_{i=0}^{N-1} \sum_{j=i+1}^{N} \delta_{ij} \| \nabla v_j \|_{U_\delta(\Gamma_{ij}) \cap \Omega_j}^2.
  \\
  &\lesssim
  N_\OO \sum_{i=0}^N \| \nabla v_i \|_{\Omega_{h,i}}^2.
\end{align}
Term $III$ may be estimated similarly to obtain
\begin{align}
  III
  &\lesssim
  N_\OO \sum_{i=1}^N \| \nabla v_i \|_{\Omega_{h,i}}^2.
\end{align}

For term $IV$, we have by recalling~\eqref{eq:term4}
\begin{align}
  IV
  &\lesssim
  N_\OO \sum_{i=1}^N h_i^{-2} \| v_i \|_{\Omega_{h,i}}^2
  +
  \sum_{i=1}^N \sum_{j=0}^{i-1} \delta_{ji} h_j^{-2} \| v_j \|_{U_\delta(\Gamma_{ij})}^2
  \\
  &\lesssim
  N_\OO \sum_{i=0}^N h_i^{-2} \| v_i \|_{\Omega_{h,i}}^2.
\end{align}

The desired estimate now follows using the standard inverse inequality~\eqref{eq:general_inverse}.
\end{proof}

\begin{lemma}[Poincar\'e inequality]
  \label{lem:poincare}
  It holds that
  \begin{align}
    \label{eq:poincare-cond}
    \| v \|_h^2 \lesssim C_P \tn v \tn_h^2 \qquad \forall v \in V_h.
  \end{align}
  where
  \begin{align}
    C_P = 1 + \max_{0\leq i \leq N} h_i^{2/d} N_{\OO_i} + \max_{0\leq i \leq N} h_i^2 N_{\OO_i}.
  \end{align}
\end{lemma}
\begin{proof}
  First note that by a Taylor expansion argument and Lemma~\ref{lem:A}, we have
\begin{align}
  \| v \|_h^2
  &\lesssim
  \sum_{i=0}^N \big( \| v_i \|^2_{\Omega_i} + h_i^2 \| \nabla v_i \|^2_{\Omega_{h,i}} \big)
  \\
  &\lesssim
  \label{eq:cond-bb}
  \sum_{i=0}^N \| v_i \|^2_{\Omega_i}
  + N_\OO \sum_{i=0}^N \left( h_i^2 \| \nabla v_i \|^2_{\Omega_i} + h_i^2 \sum_{j=i+1}^N \| \nabla (v_i - v_j) \|_{\OO_{ij}}^2 \right).
\end{align}
To control the first term on the right-hand side in~\eqref{eq:cond-bb}, let $\phi \in H^2(\Omega)$ be the solution to the dual problem
\begin{subequations}
  \begin{align}
    -\Delta \phi &= \psi \qquad \text{in }\Omega, \\
    \phi &= 0 \qquad \text{on }\partial \Omega,
  \end{align}
\end{subequations}
where $\psi \in L^2(\Omega)$. Multiplying the dual problem with $v \in V_h$ and integrating by parts, we obtain using the Cauchy-Schwarz inequality
\begin{align}
  \label{eq:cauchyschwarz0}
  \sum_{i=0}^N (v_i,\psi)_{\Omega_i}
  &=
\sum_{i=0}^N (v_i,-\Delta \phi )_{\Omega_i}
\\
&=
\sum_{i=0}^N (\nabla v_i,\nabla \phi )_{\Omega_i} - \sum_{i=1}^N \sum_{j=0}^{i-1} ([v],n_i \cdot \nabla \phi)_{\Gamma_{ij}}
\\
&\leq
\sum_{i=0}^N \|\nabla v_i\|_{\Omega_i} \|\nabla \phi \|_{\Omega_i}
\nonumber \\
&\qquad + \sum_{i=1}^N \sum_{j=0}^{i-1} (h_i  + h_i)^{-1/2} \|[v]\|_{\Gamma_{ij}}  (h_i + h_j)^{1/2} \|\nabla \phi \|_{\Gamma_{ij}}
\\
\label{eq:cauchyschwarz}
&\lesssim
\bigg( \sum_{i=0}^N \|\nabla v_i\|^2_{\Omega_i}   +
\sum_{i=1}^N \sum_{j=0}^{i-1} (h_i+h_j)^{-1} \|[v]\|^2_{\Gamma_{ij}} \bigg)^{1/2}
\nonumber \\
&\qquad
\times
\bigg(
 \|\nabla \phi \|^2_\Omega
 + \sum_{i=1}^N \sum_{j=0}^{i-1} (h_i+h_j) \|\nabla \phi \|^2_{\Gamma_{ij}} \bigg)^{1/2}.
\end{align}

Now we continue with the second factor in~\eqref{eq:cauchyschwarz}. Using the trace inequality
\begin{align}
  \label{eq:traceineq}
  \| v \|^2_{\gamma \cap K}  \lesssim h^{-1} \| v \|^2_K + h \| \nabla v \|^2_K,
\end{align}
with constant independent of the position of an interface $\gamma$ (see~\cite{Hansbo:2003aa}), we have
\begin{align}
  \label{eq:duality-temp}
  &\sum_{i=1}^N \sum_{j=0}^{i-1} h_l \|\nabla \phi \|^2_{\Gamma_{ij}}
  \nonumber \\
  &\qquad \lesssim
    \sum_{i=1}^N \sum_{j=0}^{i-1} \delta_{ji} \big(\|\nabla \phi \|^2_{\mcK_{h,i}(\Gamma_{ij})}
    + h_l^2 \|\nabla^2 \phi \|^2_{\mcK_{h,i}(\Gamma_{ij})} \big) \qquad l = i,j.
\end{align}
By the construction of $U_\delta(\Gamma_{ij})$, see~\eqref{eq:Uloc}, we have $\mcK_{h,i}(\Gamma_{ij}) \subseteq U_\delta(\Gamma_{ij})$ with $\delta \sim h_i$. Furthermore, by the H\"{o}lder inequality~\cite{Brezis11} with coefficients $r,s$ such that $1/r+1/s=1$ we have
\begin{align}
  \| \nabla \phi \|^2_{\mcK_{h,i}(\Gamma_{ij})}
  &\lesssim
  \| \nabla \phi \|^2_{U_\delta(\Gamma_{ij})}
  \\
  &=
  \|1 \cdot |\nabla \phi|^2\|_{L^1(U_\delta(\Gamma_{ij}))}
  \\
  &\leq
   \| 1 \|_{L^s(U_\delta(\Gamma_{ij}))} \| |\nabla \phi|^2 \|_{L^r(U_\delta(\Gamma_{ij}))}
  \\
  &=
   | U_\delta(\Gamma_{ij}) |^{1/s} \| |\nabla \phi|^2 \|_{L^r(U_\delta(\Gamma_{ij}))}
  \\
  &\lesssim
  h_i^{1/s} | \Gamma_{ij} |^{1/s} \| |\nabla \phi|^2 \|_{L^r(U_\delta(\Gamma_{ij}))}
  \\
  &\lesssim
  h_i^{1 - 2/p} \| \nabla \phi \|^2_{L^p(U_\delta(\Gamma_{ij}))}
  \\
  &\lesssim
  \label{eq:Holder-result}
  h_i^{1 - 2/p} \| \phi \|^2_{W^1_p(U_\delta(\Gamma_{ij}))}
\end{align}
with $p=2r$ and
$1/s
= 1 - 1 / r
= 1 - 2 / p$.

To determine $p$ in~\eqref{eq:Holder-result}, we use the Sobolev embedding $W^l_q(\Omega) \subseteq W^k_p(\Omega)$~\cite{Brezis11} with $k=1$, $l=2$ and $q=2$. This is motivated by the fact that due to elliptic regularity and $\psi \in L^2(\Omega)$, we have $\phi \in H^2(\Omega)$. Since the embedding holds for
$1/p - k/d = 1/q - l/d$ ~\cite{Brezis11},
we obtain
$p = 2d / (d - 2)$,
where $p=\infty$ for $d=2$. Thus
\begin{align}
  h_i^{1 - 2/p} \| \phi \|^2_{W^1_p(U_\delta(\Gamma_{ij}))}
  &\lesssim
    h_i^{2/d}  \| \phi \|^2_{W^2_{2d/(d-2)}(U_\delta(\Gamma_{ij}))}
    \\
  &\lesssim
    \label{eq:Sobolev-result}
  h_i^{2/d} \| \phi \|^2_{H^2(U_\delta(\Gamma_{ij}))}.
\end{align}
Cf.~\cite{Brezis11} regarding the last inequality for $d=2,3$. Returning to the second factor in~\eqref{eq:cauchyschwarz} we thus have, using~\eqref{eq:duality-temp}, \eqref{eq:Holder-result} and \eqref{eq:Sobolev-result} together with a standard duality argument (see e.g.~\cite{BreSco08}), elliptic regularity, a stability estimate, the Poincar\'e equality and ignoring higher order terms that
\begin{align}
  \label{eq:duality-result}
  &\|\nabla \phi \|^2_\Omega
    + \sum_{i=1}^N \sum_{j=0}^{i-1} (h_i+h_j) \|\nabla \phi \|^2_{\Gamma_{ij}}
    \\
  & \lesssim
    \|\nabla \phi \|^2_\Omega
    +  \sum_{i=1}^N \sum_{j=0}^{i-1} \delta_{ji} \left( h_i^{2/d} \|\phi \|^2_{H^2(U_\delta(\Gamma_{ij}))}
    + (h_i^2 + h_j^2) \|\nabla^2 \phi \|^2_{\mcK_{h,i}(\Gamma_{ij})}  \right)
  \\
  & \lesssim
    \|\psi \|^2_\Omega
    +  \sum_{i=1}^N \sum_{j=0}^{i-1}\delta_{ji}  \left( h_i^{2/d} \|\psi \|^2_{U_\delta(\Gamma_{ij})}
    + (h_i^2 + h_j^2) \|\psi \|^2_{\mcK_{h,i}(\Gamma_{ij})}  \right)
  \\
  & \lesssim
    \|\psi \|^2_\Omega
    +  \sum_{i=1}^N \left( N_{\OO_i} h_i^{2/d} \|\psi \|^2_{\Omega_i}
    + N_{\OO_i} h_i^2 \|\psi \|^2_{\Omega_i} \right)
  \\
  & \lesssim
    \sum_{i=0}^N (1 + h_i^{2/d} N_{\OO_i} + h_i^2 N_{\OO_i}) \| \psi \|_{\Omega_i}^2
  \\
  & \lesssim
    C_P \| \psi \|_\Omega^2.
\end{align}

The bound on $\sum_{i}(v_i,\psi)_{\Omega_i}$ in the left-hand side in~\eqref{eq:cauchyschwarz0} with $\psi = v$ now reads
\begin{align}
  \sum_{i=0}^N \| v_i \|_{\Omega_i}^2
  &\lesssim
  C_P \bigg( \sum_{i=0}^N \|\nabla v_i\|^2_{\Omega_i}   +
     \sum_{i=1}^N \sum_{j=0}^{i-1} (h_i+h_j)^{-1} \|[v]\|^2_{\Gamma_{ij}} \bigg)^{1/2}
    \label{eq:intermediate}
\end{align}
since $\| v \|_\Omega$ is bounded.

To conclude, recall~\eqref{eq:cond-bb} and insert~\eqref{eq:intermediate} to obtain the desired estimate
\begin{align}
  \label{eq:hnormtmp}
  \| v \|_h^2
  &\lesssim
  C_P \bigg( \sum_{i=0}^N \| \nabla v_i \|^2_{\Omega_i}   +
    \sum_{i=1}^N \sum_{j=0}^{i-1} (h_i+h_j)^{-1} \|[v]\|^2_{\Gamma_{ij}} \bigg)
    \nonumber \\
  &\qquad
    +
    N_\OO \sum_{i=0}^N \left( h_i^2 \| \nabla v_i \|^2_{\Omega_i} + h_i^2 \sum_{j=i+1}^N \| \nabla (v_i - v_j) \|_{\OO_{ij}}^2 \right)
  \\
  &\lesssim
    C_P \tn v \tn_h^2.
\end{align}
\end{proof}

\begin{theorem}[Condition number estimate]
  It holds that
  \begin{align}
    \label{eq:cond-number-bound}
    \kappa(\hatA) \lesssim C_P (1 + N_\OO)^2 \hmax^{-2}.
  \end{align}
\end{theorem}
\begin{proof}
  Since $\mcK_{h,i}$ is conforming and quasi-uniform we have the equivalence
  \begin{align}
    \| v_i \|^2_{\Omega_{h,i}} \sim h_i^d | \hatv_i |^2_{M_i} \qquad \forall v_i \in V_{h,i};
  \end{align}
  see e.g.~\cite{BreSco08}. It follows that
  \begin{align}
    \label{eq:equivalence-cond}
    \| v \|^2_h =
    \sum_{i=0}^N \| v_i \|^2_{\Omega_{h,i}}
    \sim
    \sum_{i=0}^N h_i^d | \hatv_i |^2_{M_i}
    \sim
    \hmax^d | \hatv |^2_{M}.
  \end{align}
  Recall the definition of the matrix norm
  \begin{align}
    \label{eq:def-matrixnorm}
    |\hatA|_M = \sup_{\hatv \neq 0} \frac{|\hatA \hatv|_M}{|\hatv|_M}.
  \end{align}
  To estimate $| \hatA |_M$, we use the definition of the stiffness matrix~\eqref{eq:def-stiffness-matrix}, the inverse inequality~\eqref{eq:inverse-cond} and the equivalence~\eqref{eq:equivalence-cond} to obtain
  \begin{align}
    | \hatA \hatv |_M
    &=
    \sup_{\hatw \neq 0} \frac{(\hatA \hatv, \hatw)_M}{|\hatw|_M}
    \\
    &=
    \sup_{\hatw \neq 0} \frac{A_h(v,w)}{|\hatw|_M}
    \\
    &\lesssim
    \sup_{\hatw \neq 0} \frac{\tn v \tn_h \tn w \tn_h}{|\hatw|_M}
    \\
    &\lesssim
    \sup_{\hatw \neq 0} \displaystyle \frac{\displaystyle (1 + N_\OO) \hmax^{-2} \| v \|_h \| w \|_h}{|\hatw|_M}
    \\
    &\lesssim
    (1 + N_\OO) \hmax^{d-2} | \hatv |_M.
  \end{align}
  Dividing by $|\hatv|$ and using the definition of the matrix norm~\eqref{eq:def-matrixnorm} yields
  \begin{align}
    \label{eq:A-est}
    | \hatA |_M \lesssim (1 + N_\OO) \hmax^{d-2}.
  \end{align}
  To estimate $| \hatA^{-1} |_M$, we proceed similarly, and additionally use the Poincar\'e inequality~Lemma~\ref{lem:poincare} and the coercivity of the bilinear form~\eqref{eq:coercive} to obtain
  \begin{align}
    \hmax^d |\hatv|_M^2
    &\sim
    \| v \|_h^2
    \\
    &\lesssim
      C_P \tn v \tn_h^2
    \\
    &\lesssim
      C_P A_h(v,v)
    \\
    &=
      C_P (\hatA \hatv, \hatv)_M
    \\
    &\leq
      C_P |\hatA \hatv|_M |\hatv|_M.
  \end{align}
  The inequality thus reads
  \begin{align}
    \hmax^d | \hatv |_M \lesssim C_P | \hatA \hatv |_M.
  \end{align}
  Setting $\hatv = \hatA^{-1} \hatw$ yields
  \begin{align}
    \hmax^d |\hatA^{-1} \hatw|_M \lesssim C_P |\hatw|_M.
  \end{align}
  Dividing by $|\hatw|_M$ and using the definition of the matrix norm~\eqref{eq:def-matrixnorm} now gives
  \begin{align}
    \label{eq:Ainv-est}
    | \hatA^{-1} |_M \lesssim C_P (1 + N_\OO) \hmax^{-d}.
  \end{align}
  By using~\eqref{eq:A-est} and~\eqref{eq:Ainv-est} in the definition of the condition number~\eqref{eq:def-condition-number}, we obtain the desired estimate~\eqref{eq:cond-number-bound}.
\end{proof}

The estimate for the condition number is supported by the numerical results presented in Figure~\ref{fig:condno_estimate}. The slope is found to be $-1.76$. The details on this example is found in~\cite{mmfem-1}.
\begin{figure}
  \begin{center}
    \includegraphics[width=0.5\textwidth]{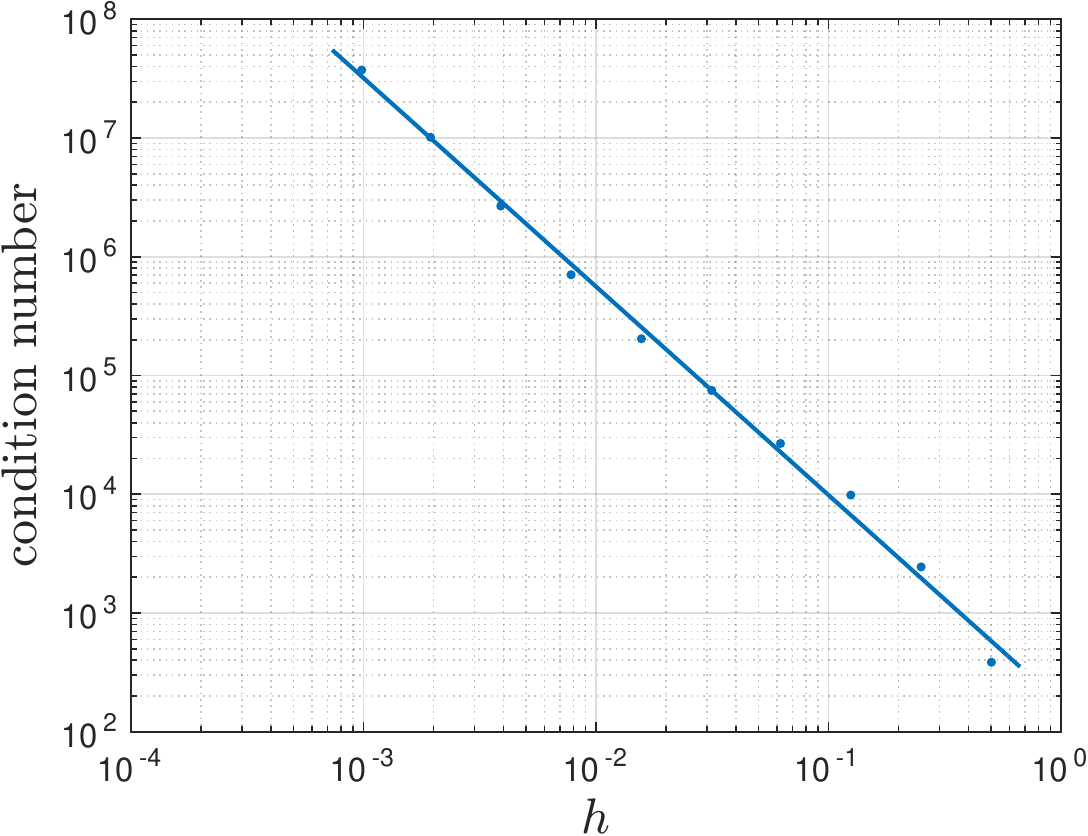}
  \end{center}
  \caption{Condition number as a function of the mesh size $h$~\cite{mmfem-1}. }
  \label{fig:condno_estimate}
\end{figure}

%---------------------------------------------------------------------------
\section{Numerical results}
\label{sec:numres}

To demonstrate the applicability and robustness of the multimesh finite element formulation, we present here a couple of numerical examples. For additional examples, we refer to the companion paper~\cite{mmfem-1}.

\subsection{Convergence under variable mesh size}

For the first example, we construct two multimesh configurations $I$ and $II$, each consisting of three parts (overlapping meshes) as show in Figure~\ref{fig:vary-mesh-size-setup}. We consider a simple Poisson problem with analytical solution
\begin{align}
  u(x,y) = \sin (\pi x) \sin (\pi y).
\end{align}
The goal is to study the convergence under refinement of the three meshes for each of the two test cases. Starting from initial coarse meshes with equal mesh sizes, we refine each part separately, using $8$ different mesh sizes, and compute the $L^2(\Omega)$ and $H^1_0(\Omega)$ error norms. A piecewise linear finite element basis is used for both configurations.

The refinement procedure is as follows. First we will refine part $0$ in $8$ steps, then part $1$ in $8$ steps, and finally part $2$ in $8$ steps. Then we swap the order and refine part $1$ first, followed by parts $0$ and $2$. We do this for all permutations of the order of the parts; in total there are $3!$ combinations. This procedure is performed for both $I$ and $II$.

Configuration $I$ is a nested configuration (but not hierarchical). Configuration $II$ is generated by placing the second and third parts in a ``random`` position on top of the background mesh of $\Omega_0 = [0,1]^2$. Specifically, we have
\begin{align}
  \Omega_1^I &= [0.2, 0.8]^2, \\ \Omega_2^I &= [0.4, 0.6]^2, \\
  \Omega_1^{II} &= [0.2, 0.8] \times [0.3, 0.75], \text{ rotated }23^{\circ}, \\ \Omega_2^{II} &= [0.3, 0.5] \times [0.05, 0.8], \text{ rotated }44^{\circ},
\end{align}
as illustrated in Figure~\ref{fig:vary-mesh-size-setup}. The meshes are refined with mesh sizes $2^{-k}$, $k=3, \ldots, 10$, i.e., $8$ steps. Thus, the mesh size ratio between two parts are in this example at most $2^7=128$.

\begin{figure}
  \centering
  \includegraphics[width=0.45\textwidth]{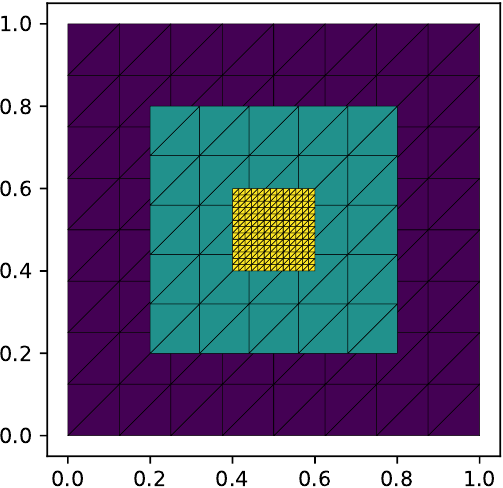}
  \includegraphics[width=0.45\textwidth]{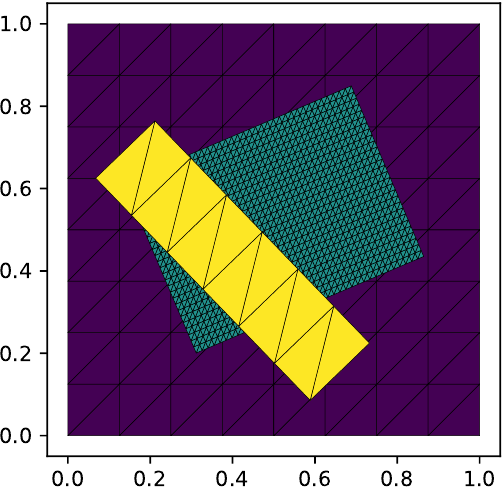}
  \caption{Configuration $I$ (left) and and configuration $II$ (right) exemplified by refined part 2 for $I$ and refined part 1 for $II$. The coarse and fine meshes have mesh sizes $2^{-3}$ and $2^{-6}$ respectively.}
  \label{fig:vary-mesh-size-setup}
\end{figure}

In Figure~\ref{fig:vary-mesh-size-errors} we show the $L^2(\Omega)$ and $H^1_0(\Omega)$ errors for the two configurations. As expected, the different curves start and end in the same point. Moreover, we see that during refinement of the first part, errors decrease but flatten. This is due to the fact that the errors from the other two, unrefined, parts dominate. When the second part starts being refined, the errors drop but will again flatten since the errors are dominated by the third and last unrefined part. Refining this part results in a sharp decrease in the error. Due to the effect of dominating errors from different parts, we observe two L-shaped drops for each refinement permutation, for both configurations and for both error quantities.

There is no significant L-shape decrease for the first refined part, but it would be possible to construct a multimesh configuration such that this would be the case. For the example with this analytical solution, the first part would have a dominating error if the area of the part would be dominating.

It is worth noting is that the errors decrease smoothly and the method is stable despite the large differences in mesh size.

\begin{figure}
  \centering
  \includegraphics[width=0.45\textwidth]{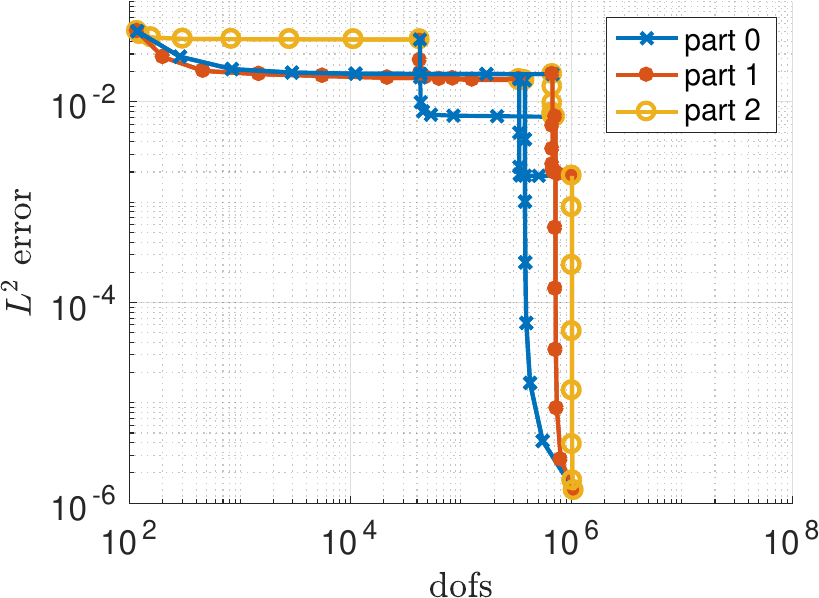}
  \includegraphics[width=0.45\textwidth]{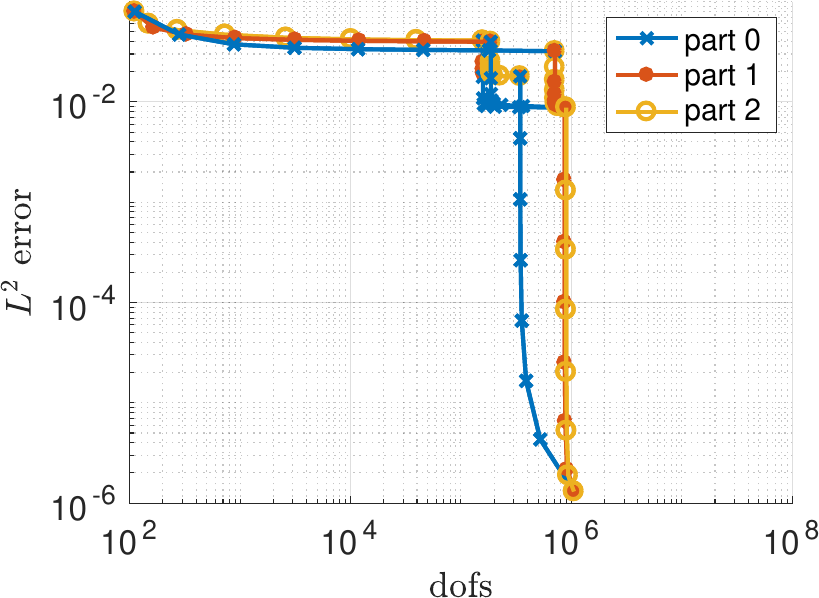}\\
  \includegraphics[width=0.45\textwidth]{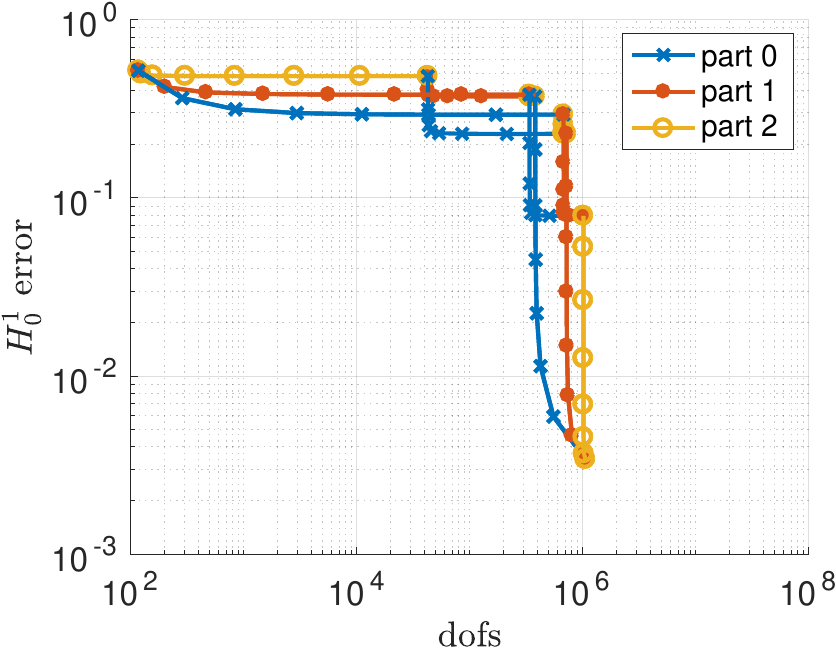}
  \includegraphics[width=0.45\textwidth]{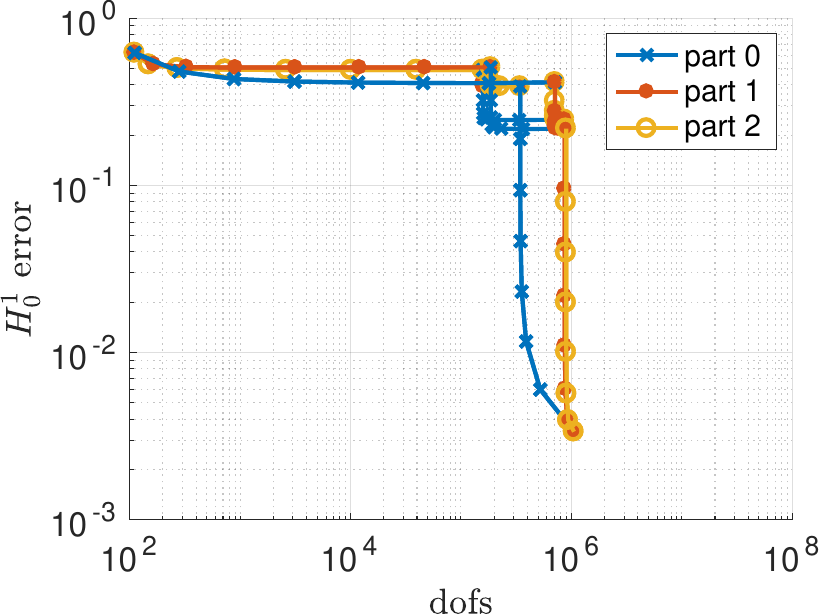}
  \caption{Errors during refinement of multimesh configurations $I$ (left column) and $II$ (right column). Colors and markers indicate which part is being refined. The refinement procedure starts with all parts having a mesh size of $2^{-3}$ resulting in approximately $10^2$ degrees of freedom. Each part is then refined individually and sequentially as described in the text, until all parts have a mesh size of $2^{-10}$, resulting in a total of approximately $10^6$ degrees of freedom. }
  \label{fig:vary-mesh-size-errors}
\end{figure}

\subsection{Boundary layer resolution}

To demonstrate the potential of the multimesh formulation for local adaptation, we consider the boundary value problem
\begin{subequations}
  \label{eq:blayer}
  \begin{align}
    -\Delta u + \epsilon^{-2} u &= f \qquad \text{in } \Omega,\\
    \qquad u &= 0 \qquad \text{on } \Gamma_0, \\
    \qquad u &= 1 \qquad \text{on } \Gamma_1.
  \end{align}
\end{subequations}
For $\epsilon\rightarrow 0$, the PDE reduces to $u = 0$ which is compatible with the boundary condition on $\Gamma_0$. As a consequence, the solution for small $\epsilon$ is $u \approx 0$ away from the boundary $\Gamma_1$ and then an exponential transition to $u = 1$ close to $\Gamma_1$. The width of the boundary layer is $\sim \epsilon$. The multimesh finite element formulation is identical to \eqref{eq:method} with the additional term
\begin{align}
  \epsilon^{-2} \sum_{i=0}^N (v_i,w_i)_{\Omega_i}.
\end{align}

We consider a model problem where the domain $\Omega$ is defined by $[0,1]^2 \setminus \omega$, where $\omega$ is the shape of the standard NACA~6409 standard. We let $\Gamma_0$ be the boundary of the unit square and let $\Gamma_1$ be the boundary of the airfoil. The solution exhibits a boundary layer of width $\epsilon$ on the airfoil boundary.

To discretize the problem, we let $\hatmcK_{h,0}$ be a uniform mesh of the unit square with mesh size $H = 2^{-(6 + k)}$ and let $\hatmcK_{h,1}$ be a boundary-fitted mesh of width $w = 0.1\cdot 2^{-k}$ for $k = 0,1,2,3,4$. The boundary layer parameter is chosen as $\epsilon = w/2$. The mesh size $h \ll H$ is chosen to well resolve the boundary layer. Note that we intentionally take $w$ small relative to the boundary layer width so as not to get the entire boundary layer transition on the finer mesh, in order to illustrate better the robustness of the method and the coupling of the solution represented on the background mesh and the boundary-fitted mesh on the interface $\Gamma$. If instead we take $\epsilon = w/10$, the solution would transition quickly to $u\approx 0$ on the interface $\Gamma$.

Figure~\ref{fig:blayer-2d} shows the solution for $k=1,2,3,4$, clearly demonstrating the decreasing width of the boundary layer with increasing $k$. Note the smooth transition of the solution going from the representation on the coarse background mesh to the fine boundary-fitted mesh. In Figure~\ref{fig:blayer-3d}, a 3D view is plotted for both solution components for $k = 0,1,2,3,4$. Finally, Figure~\ref{fig:blayer-zoom} shows detailed plots of the solution close to the boundary layer for $k = 0$ and $k = 4$.

\begin{figure}[htbp]
  \centering
  \includegraphics[width=0.45\textwidth]{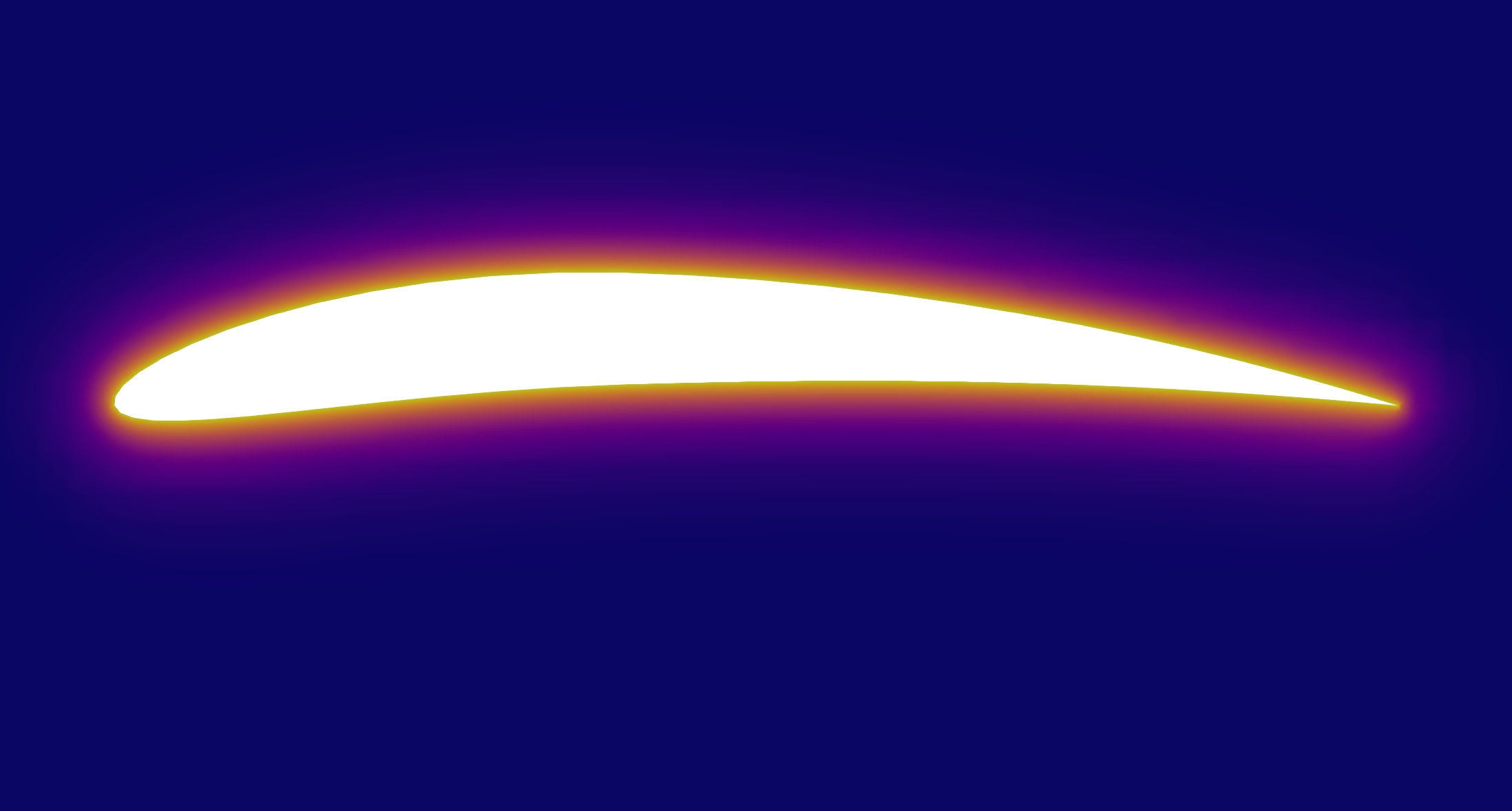}
  \includegraphics[width=0.45\textwidth]{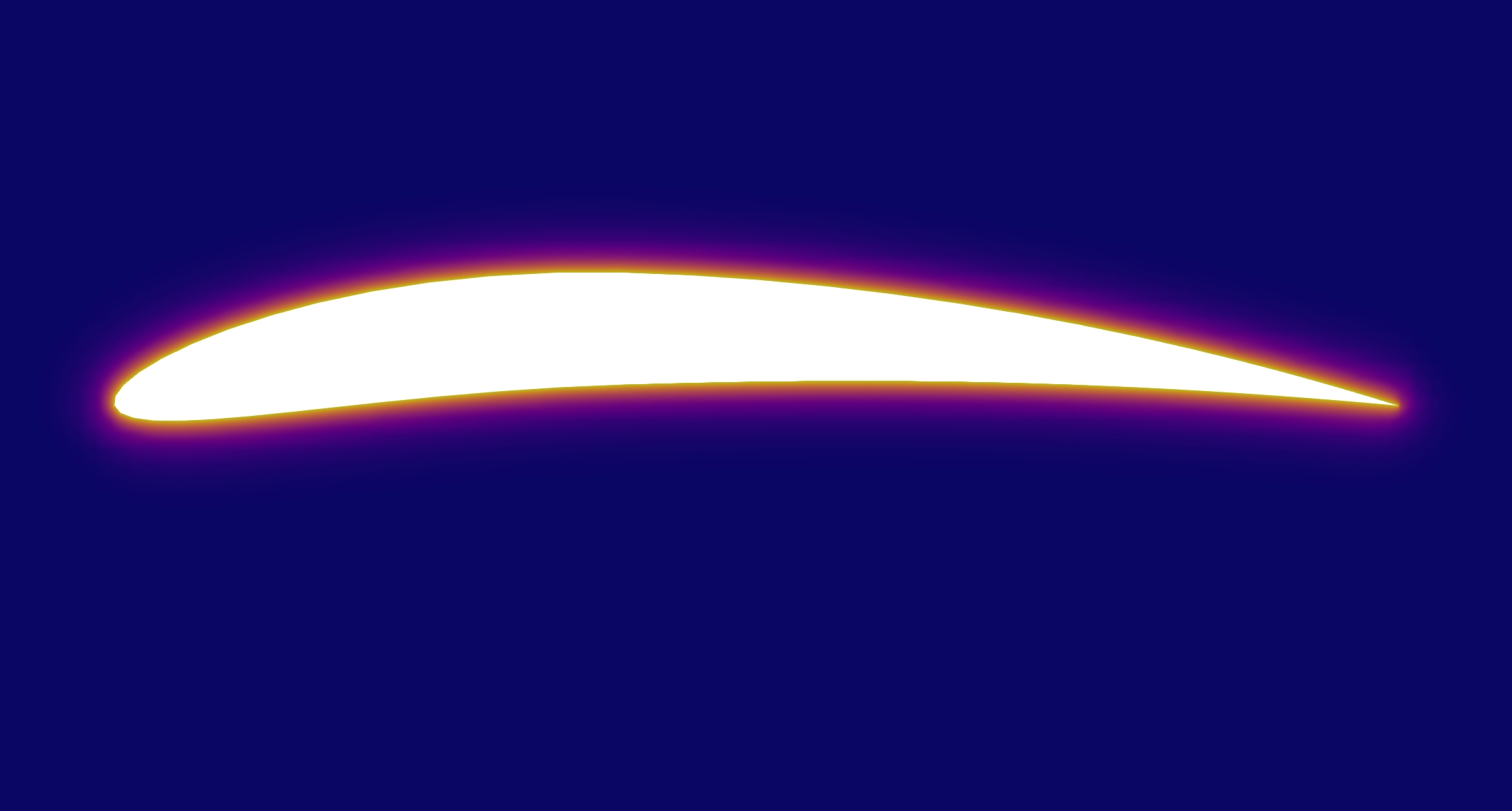} \\[1mm]
  \hspace{0mm}
  \includegraphics[width=0.45\textwidth]{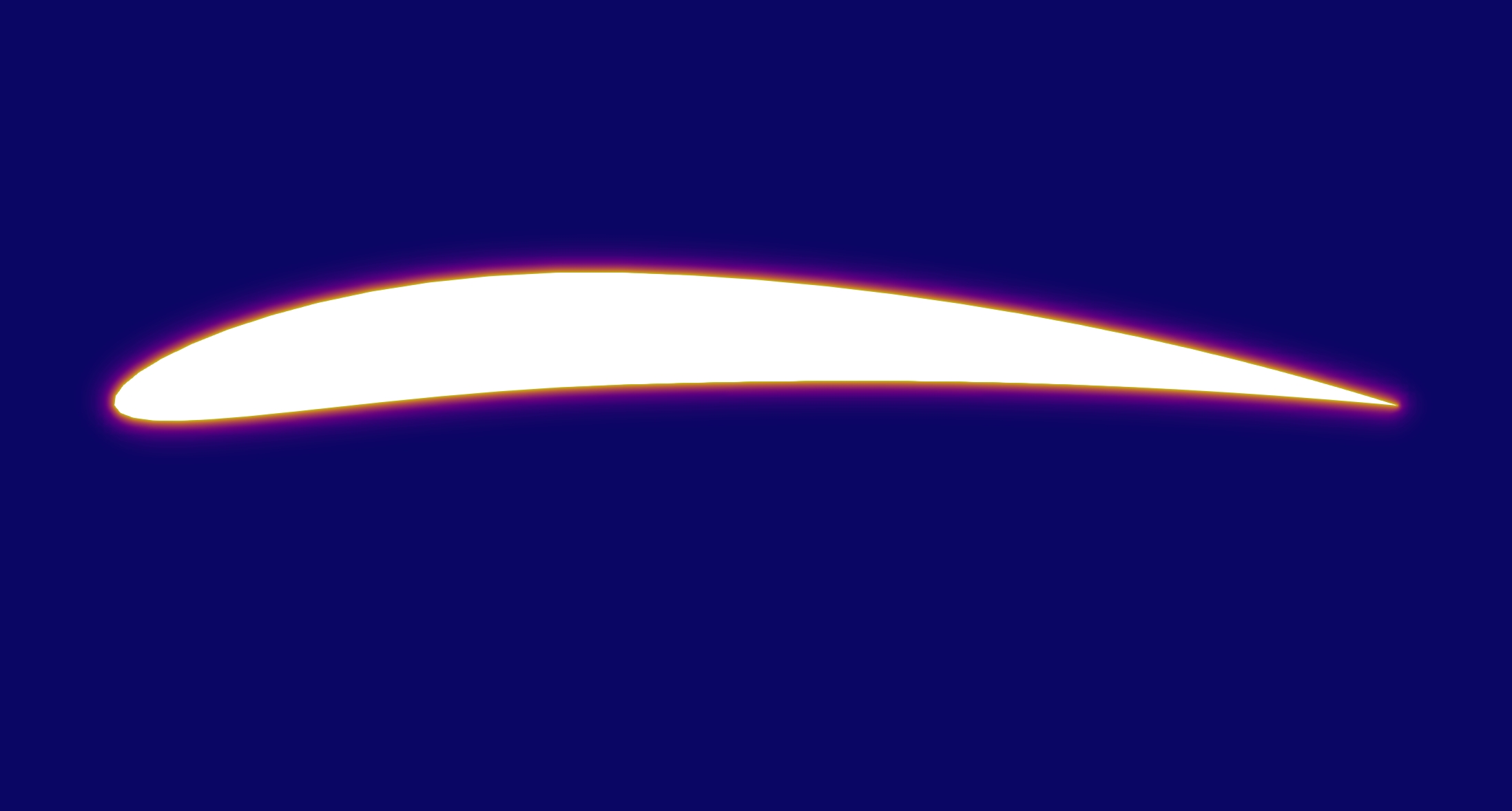}
  \includegraphics[width=0.45\textwidth]{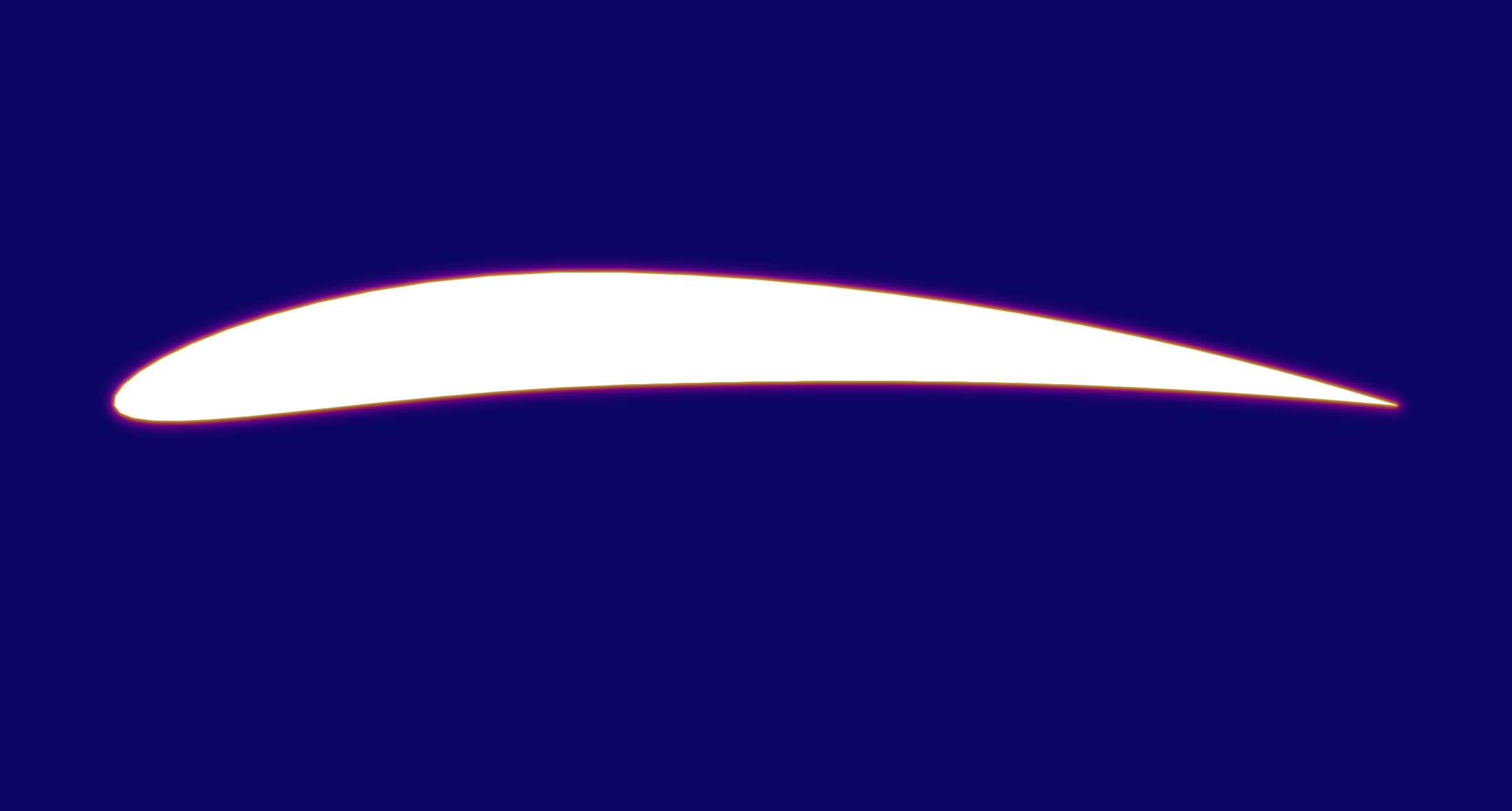}
  \caption{Solution of the boundary layer problem~\eqref{eq:blayer} for $k=1,2,3,4$}
  \label{fig:blayer-2d}
\end{figure}

\begin{figure}[htbp]
  \centering
  \includegraphics[width=0.3\textwidth]{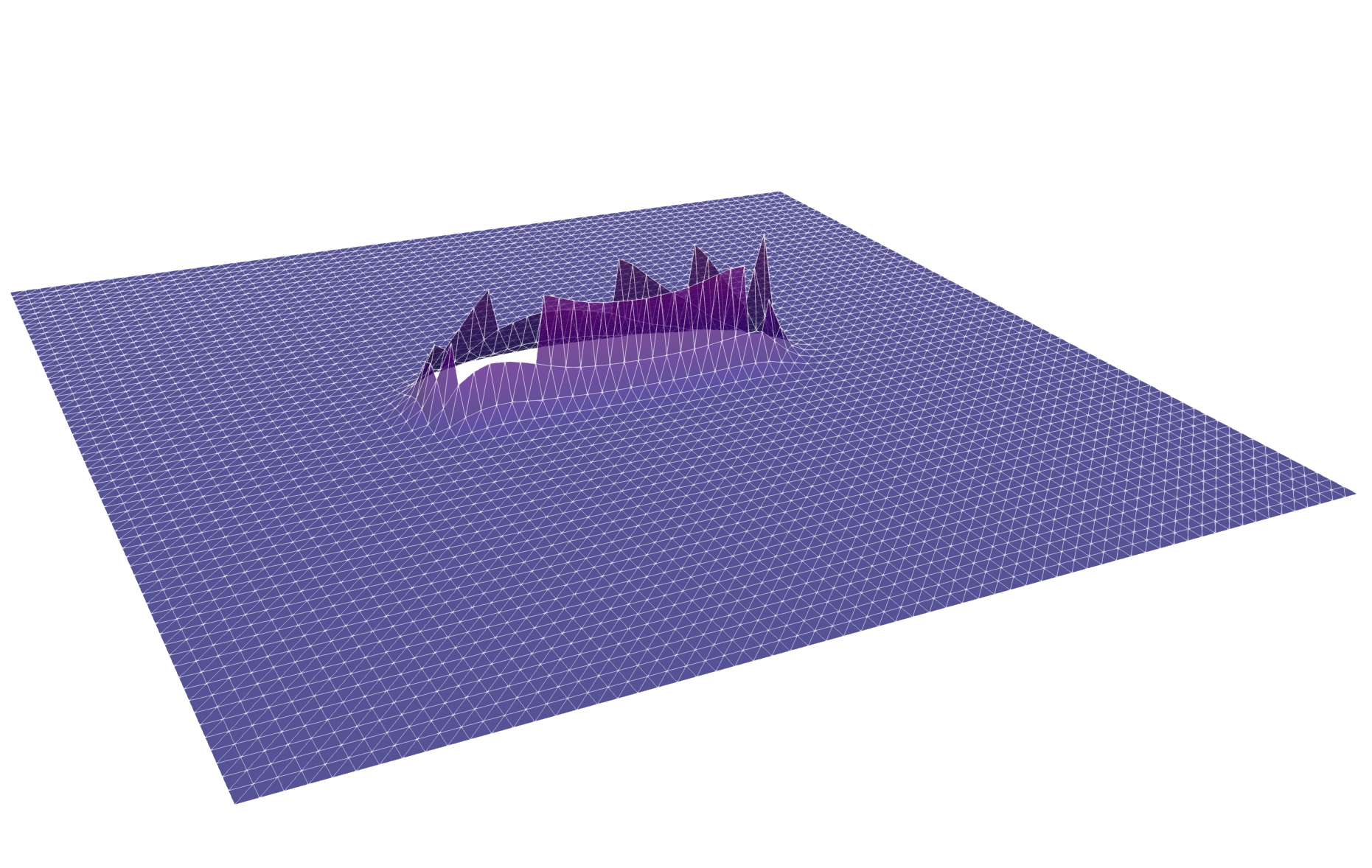}
  \includegraphics[width=0.3\textwidth]{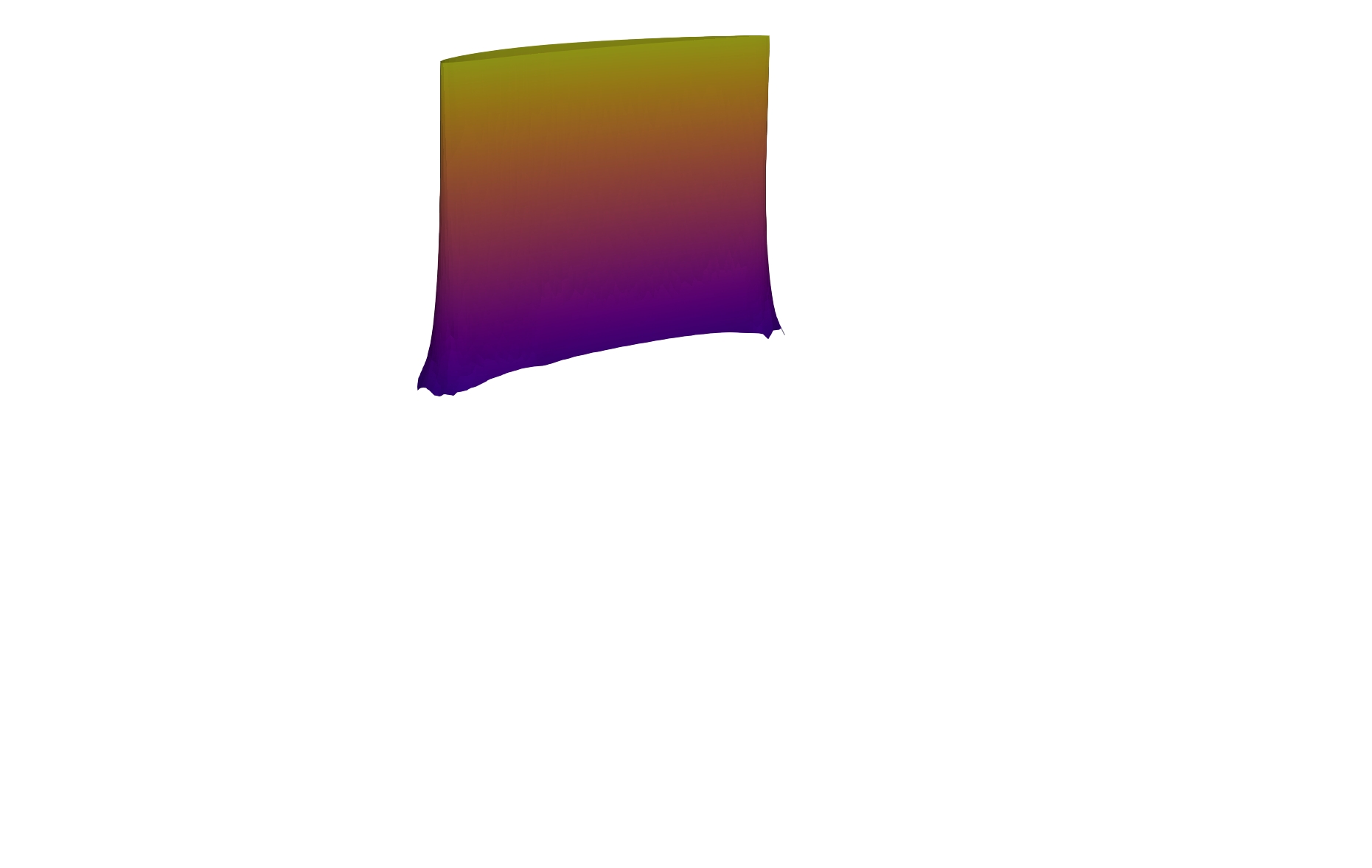}
  \includegraphics[width=0.3\textwidth]{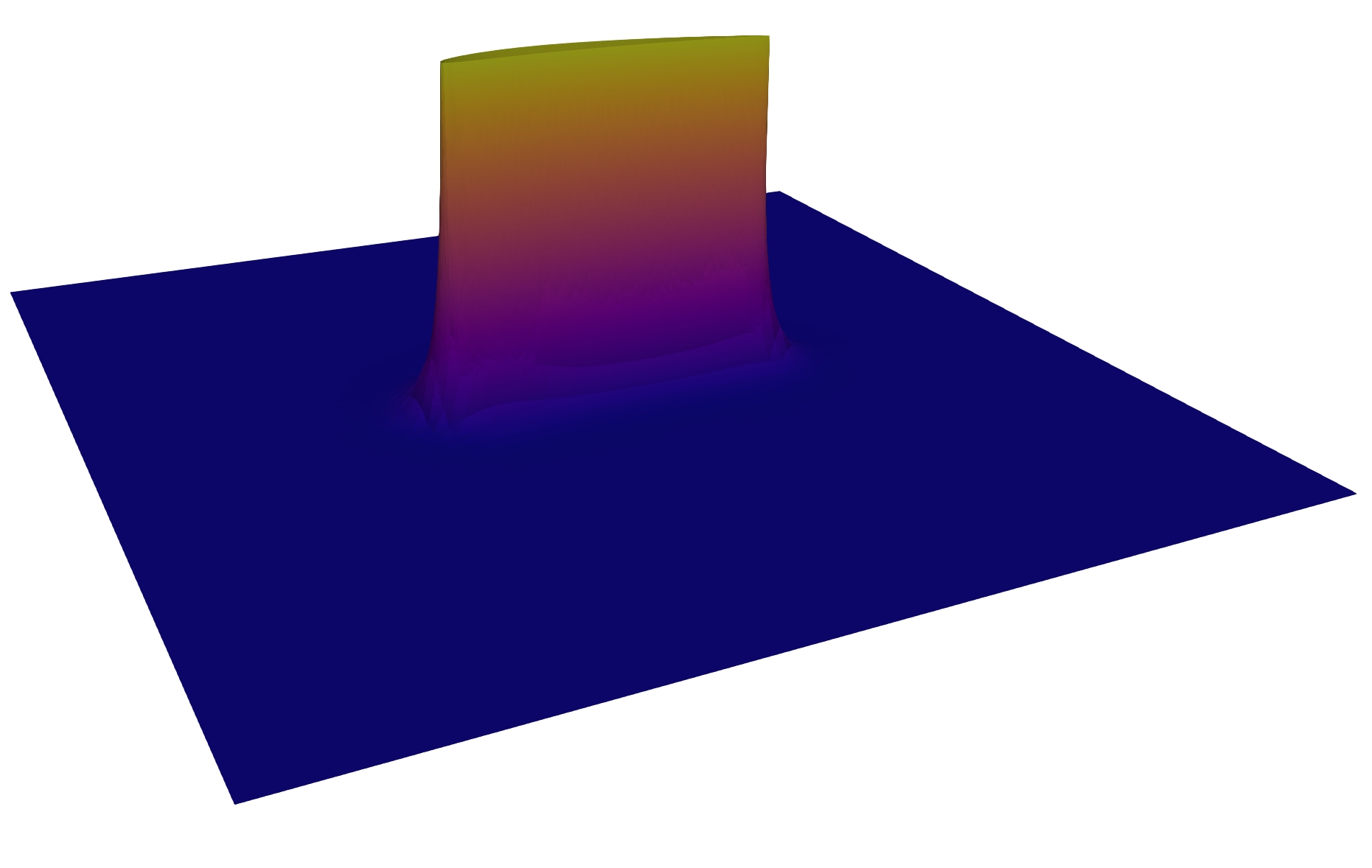} \\
  \includegraphics[width=0.3\textwidth]{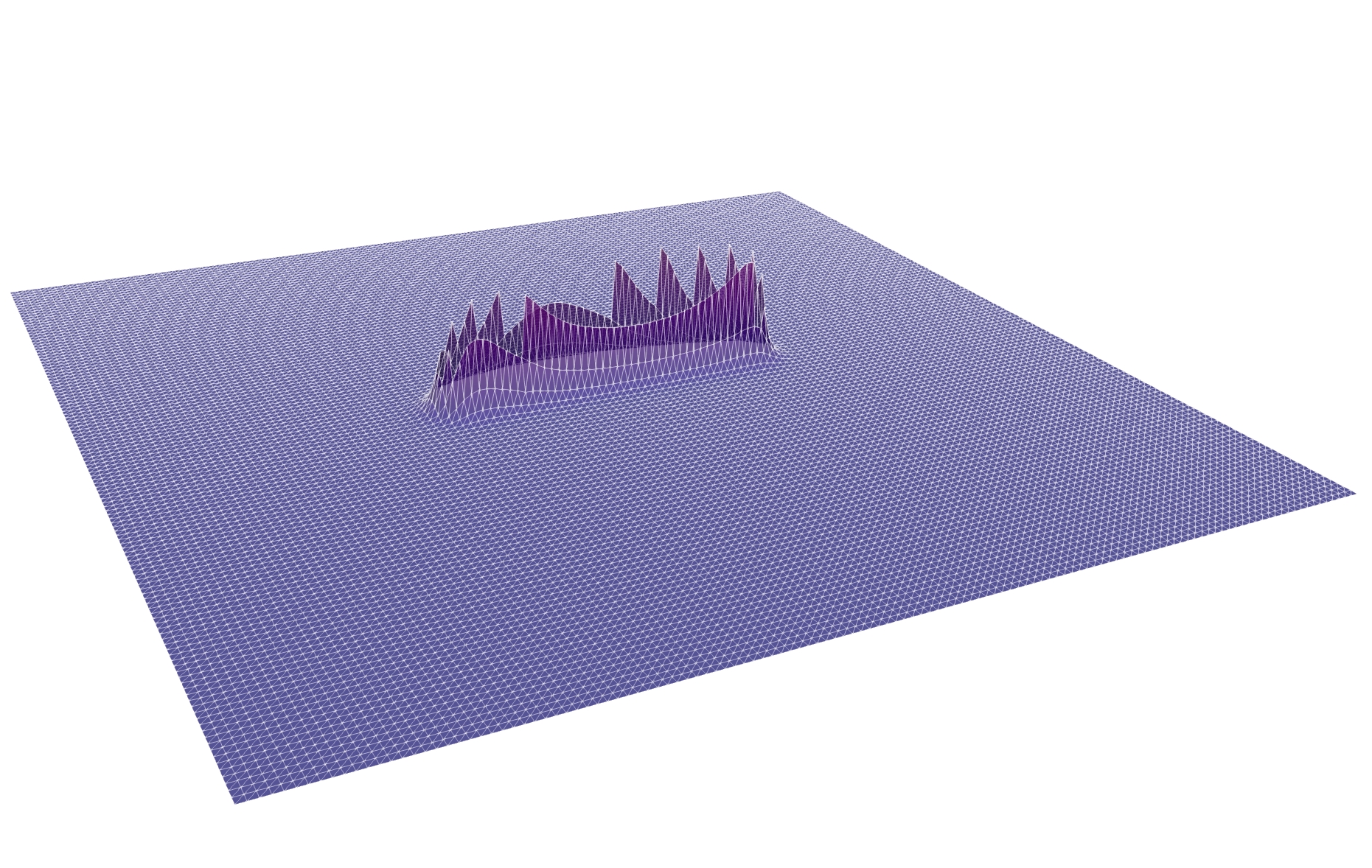}
  \includegraphics[width=0.3\textwidth]{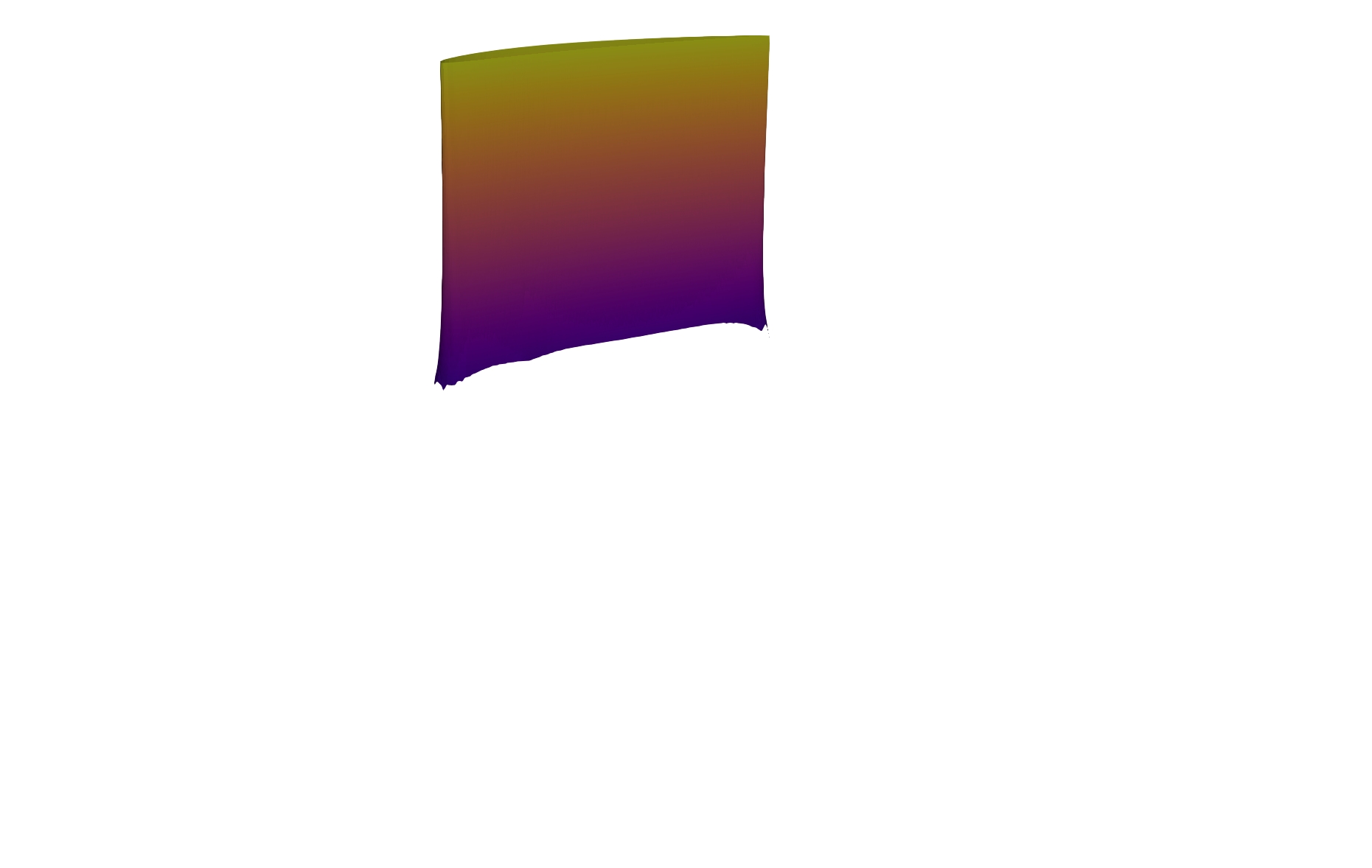}
  \includegraphics[width=0.3\textwidth]{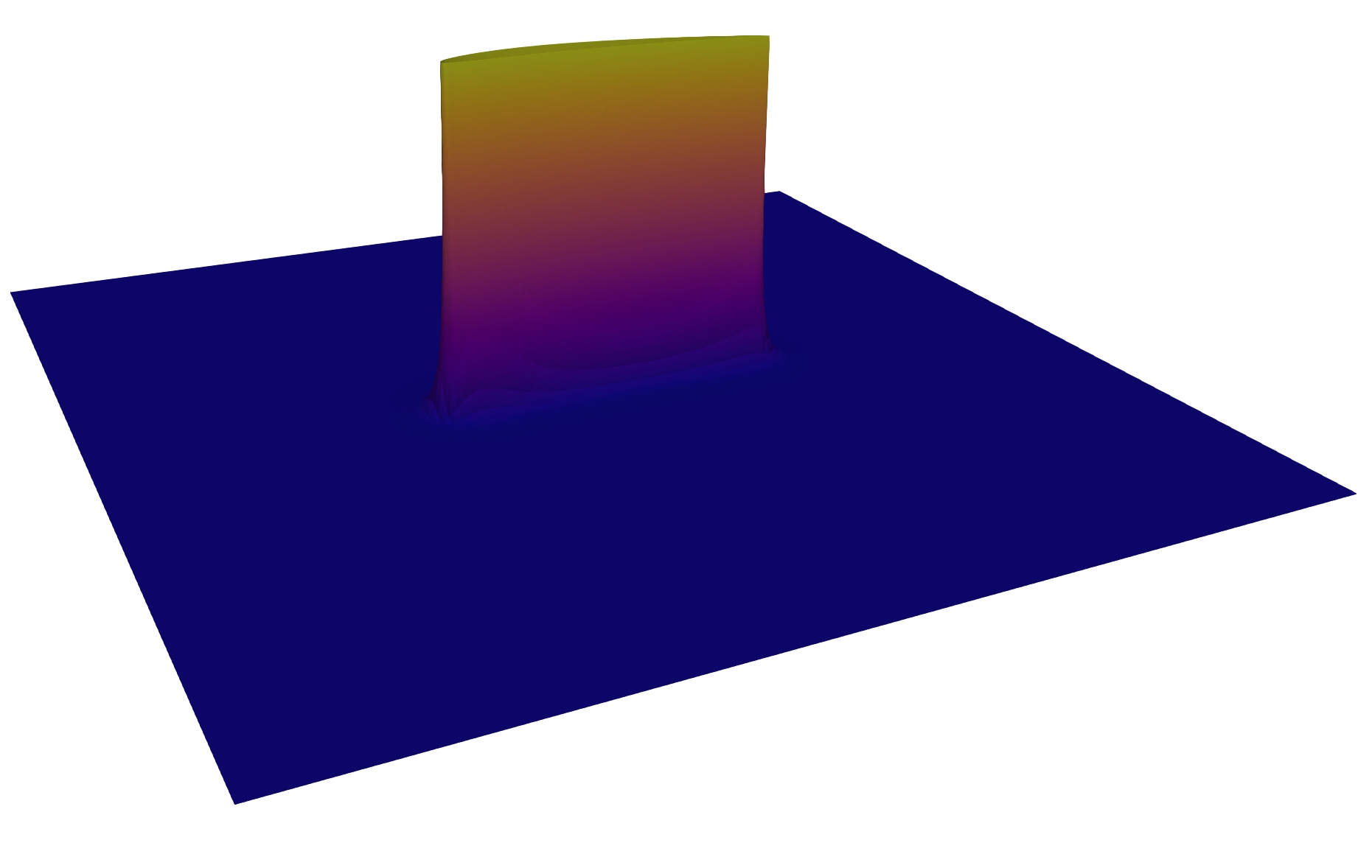} \\
  \includegraphics[width=0.3\textwidth]{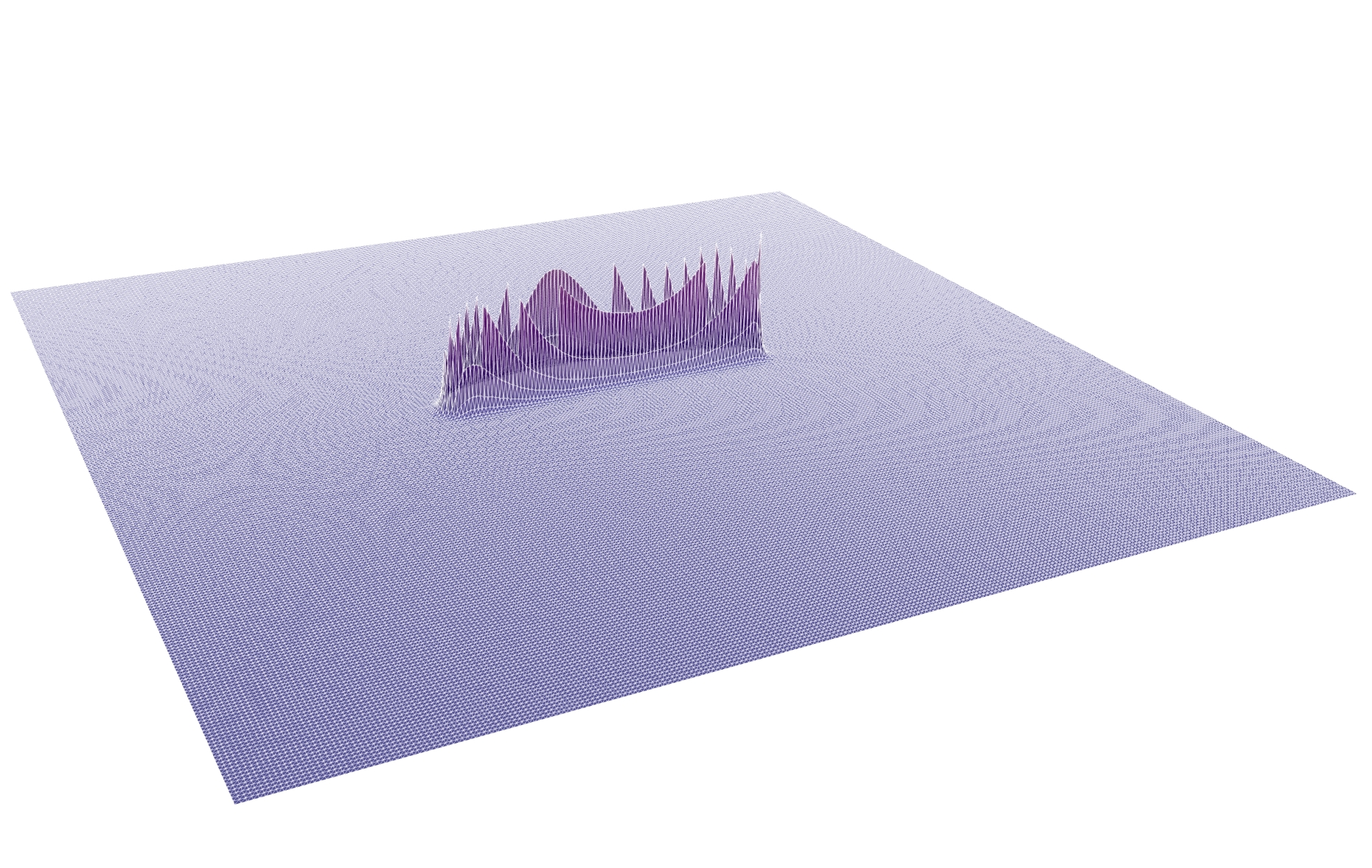}
  \includegraphics[width=0.3\textwidth]{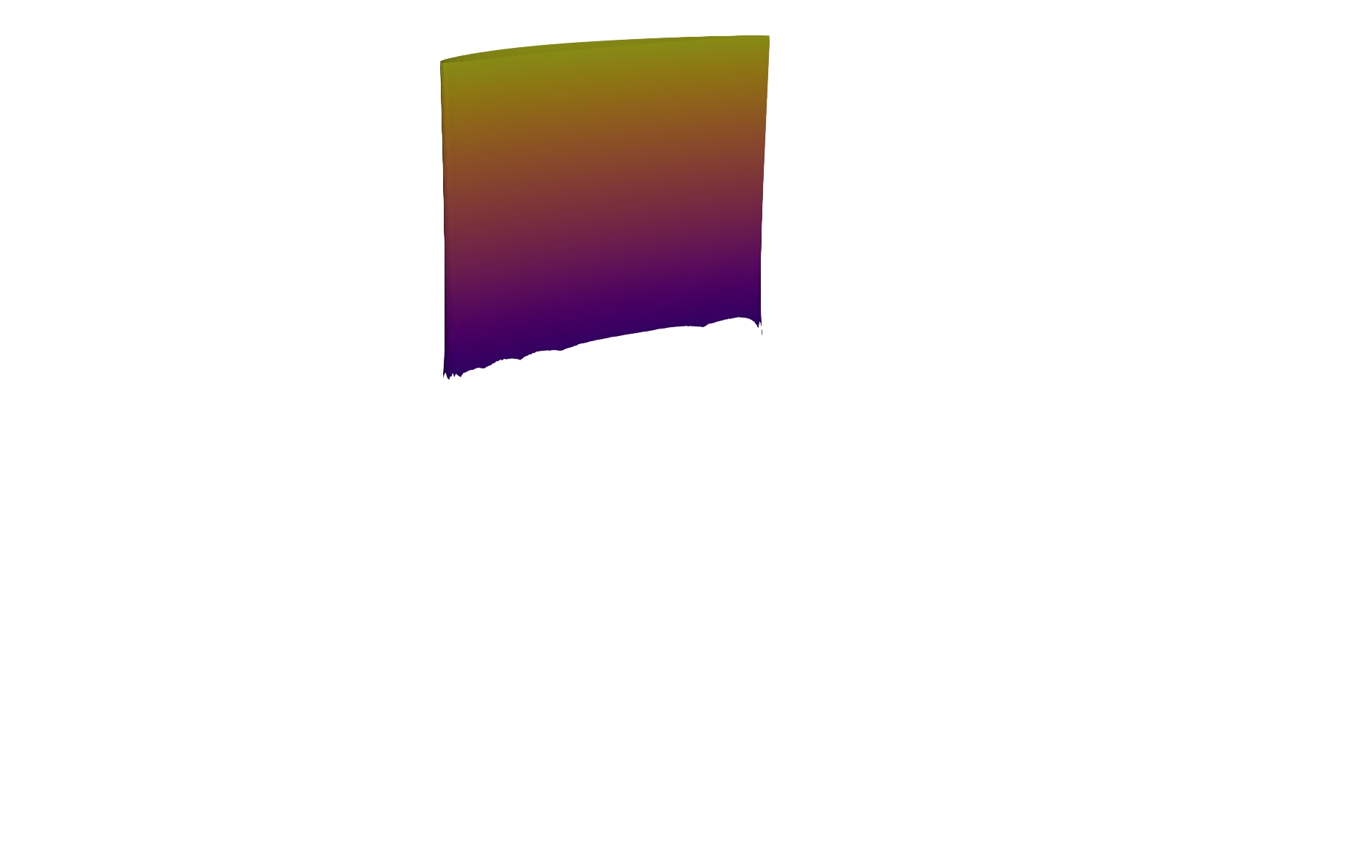}
  \includegraphics[width=0.3\textwidth]{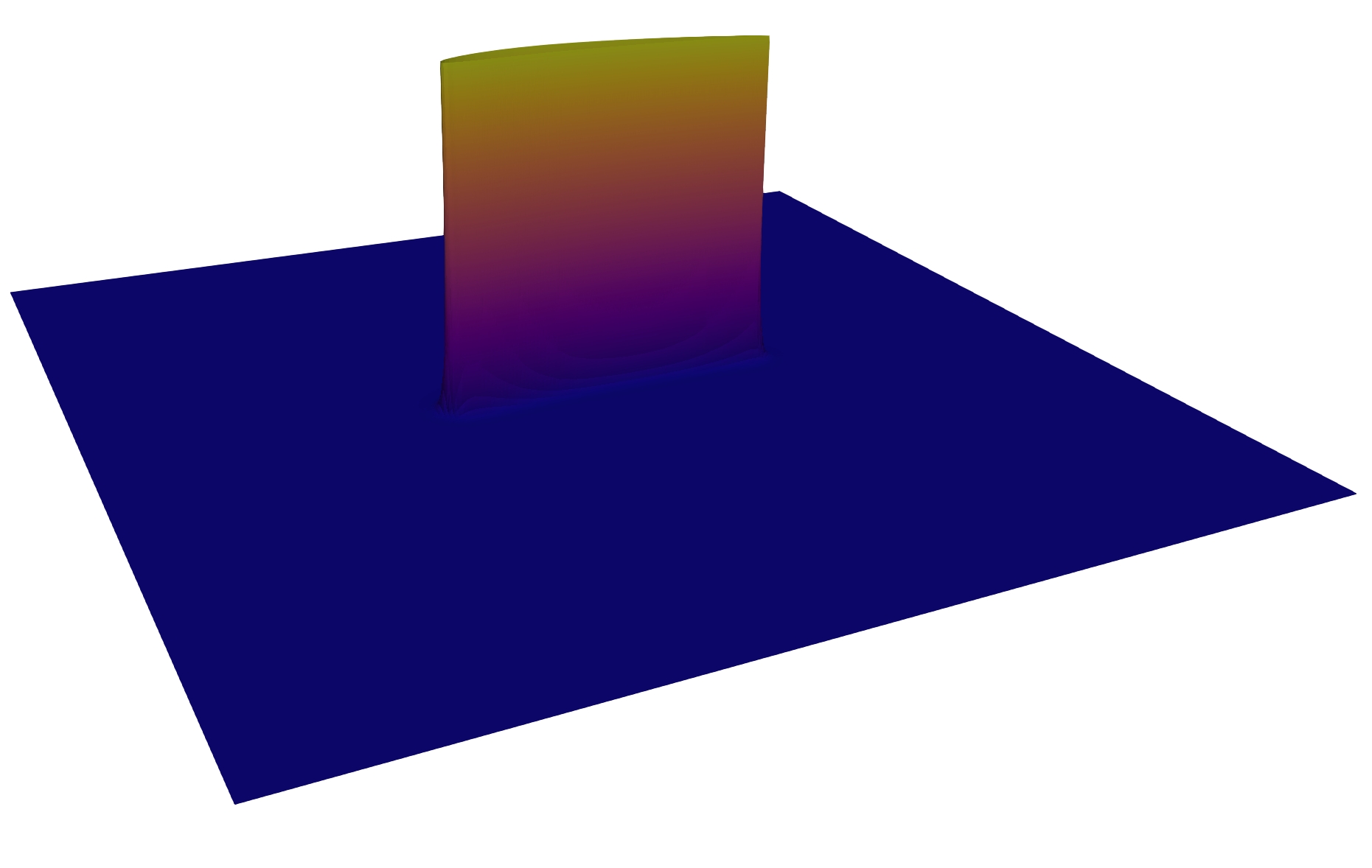} \\
  \includegraphics[width=0.3\textwidth]{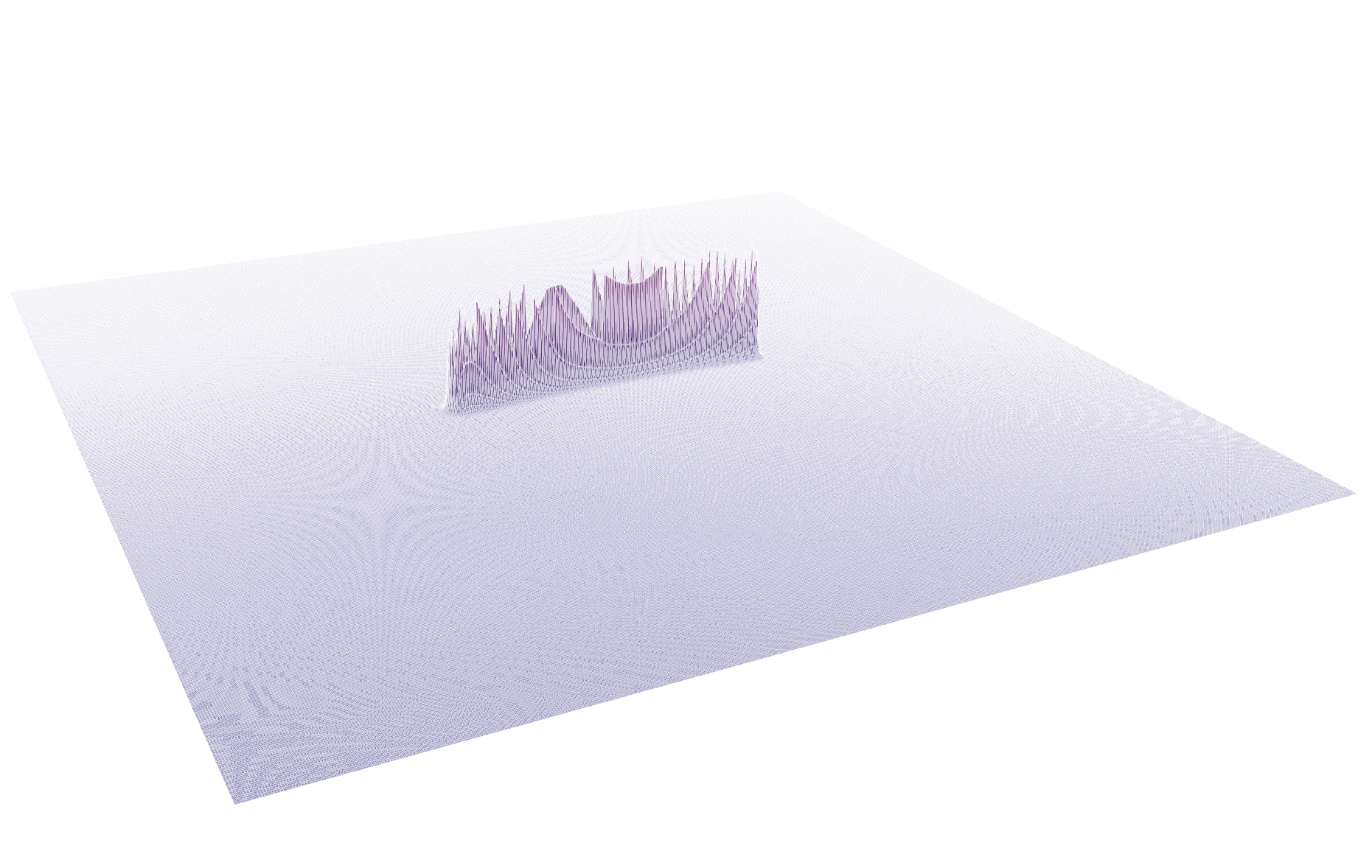}
  \includegraphics[width=0.3\textwidth]{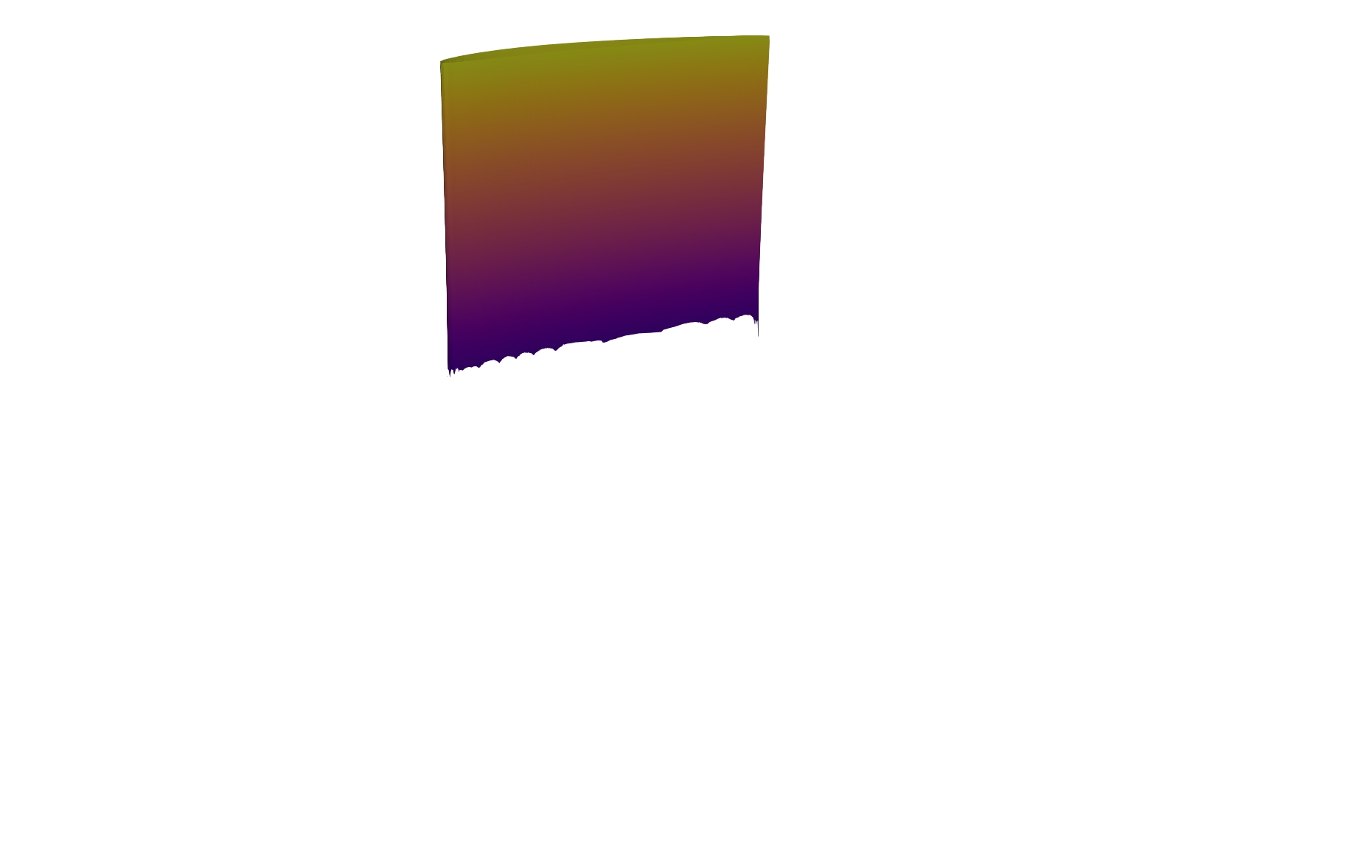}
  \includegraphics[width=0.3\textwidth]{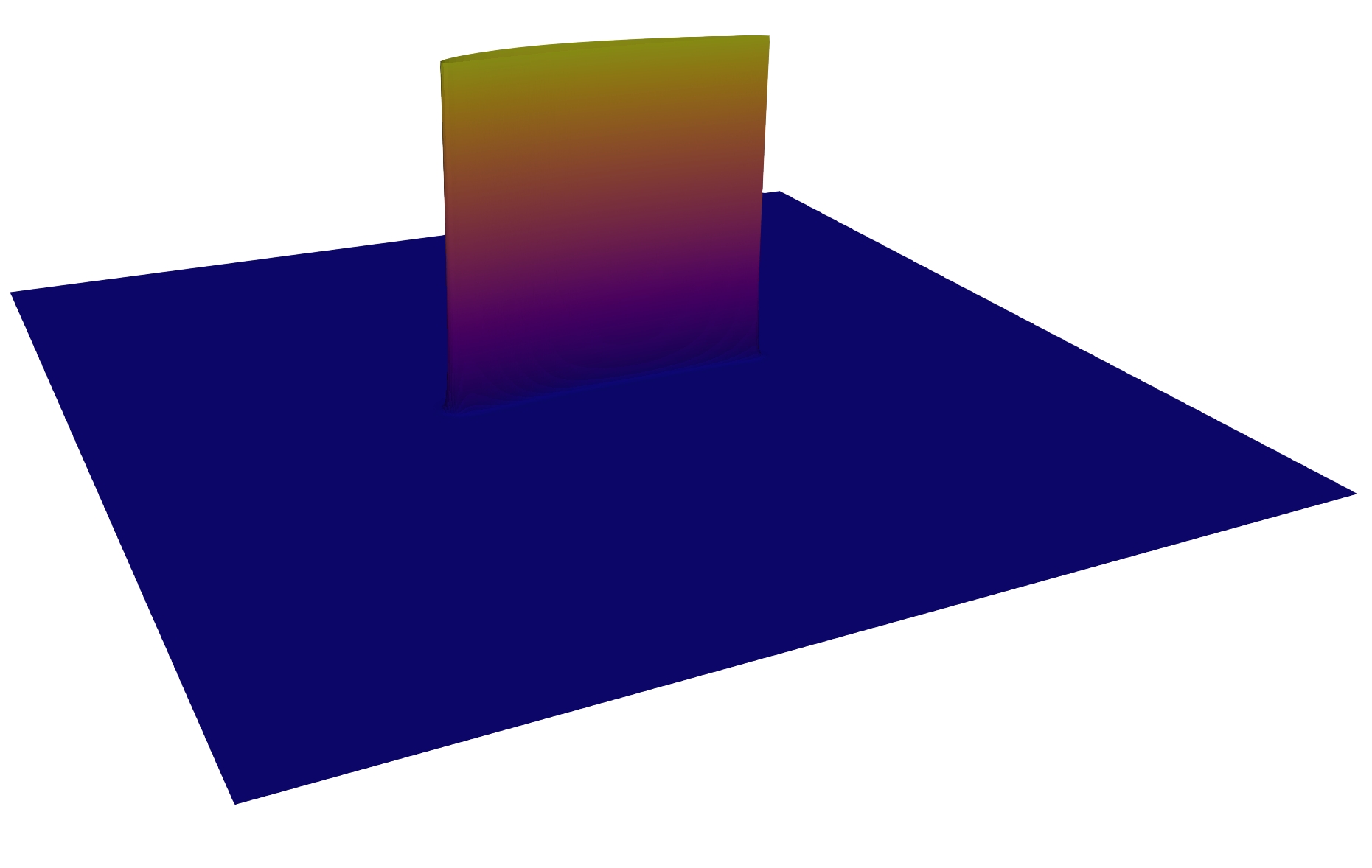} \\
  \includegraphics[width=0.3\textwidth]{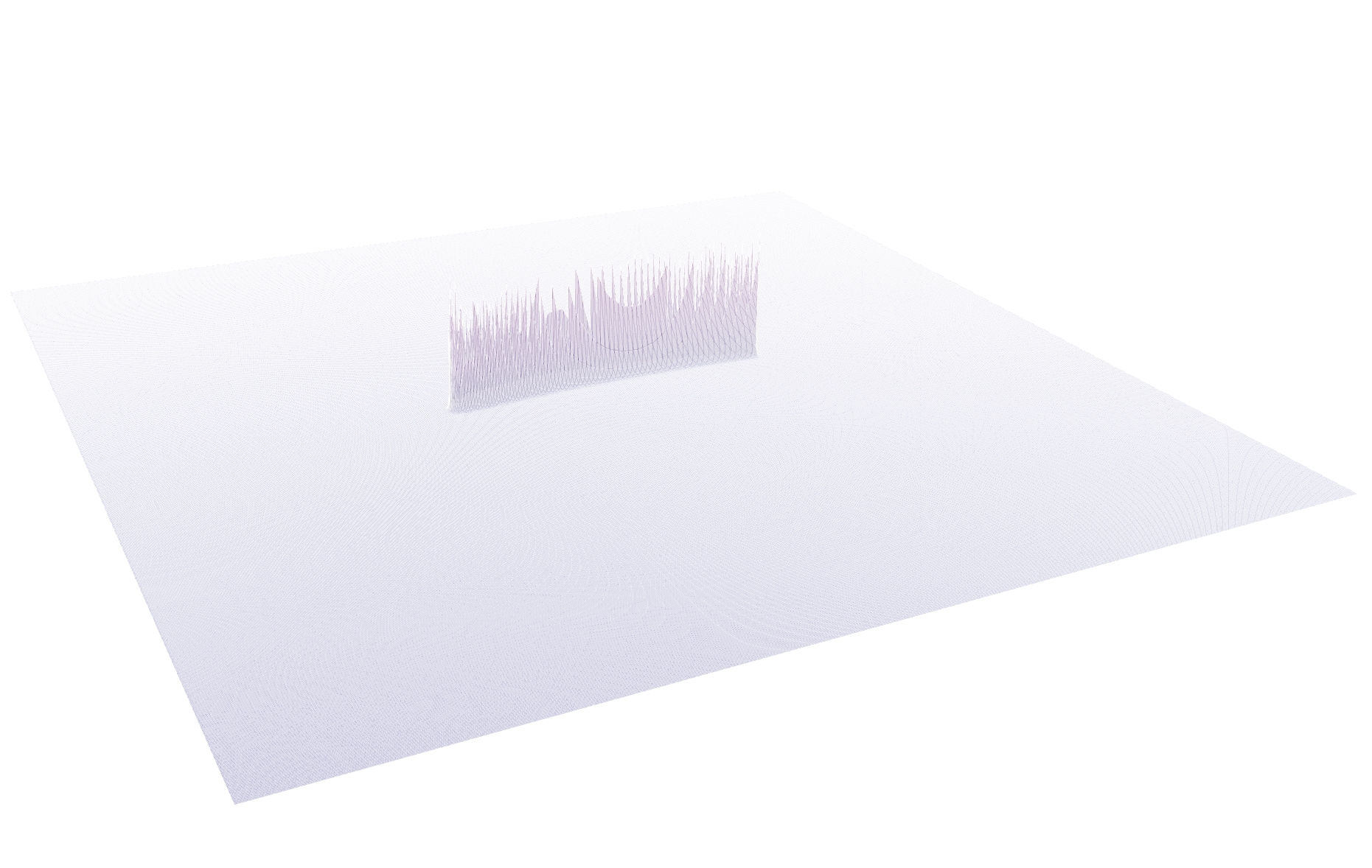}
  \includegraphics[width=0.3\textwidth]{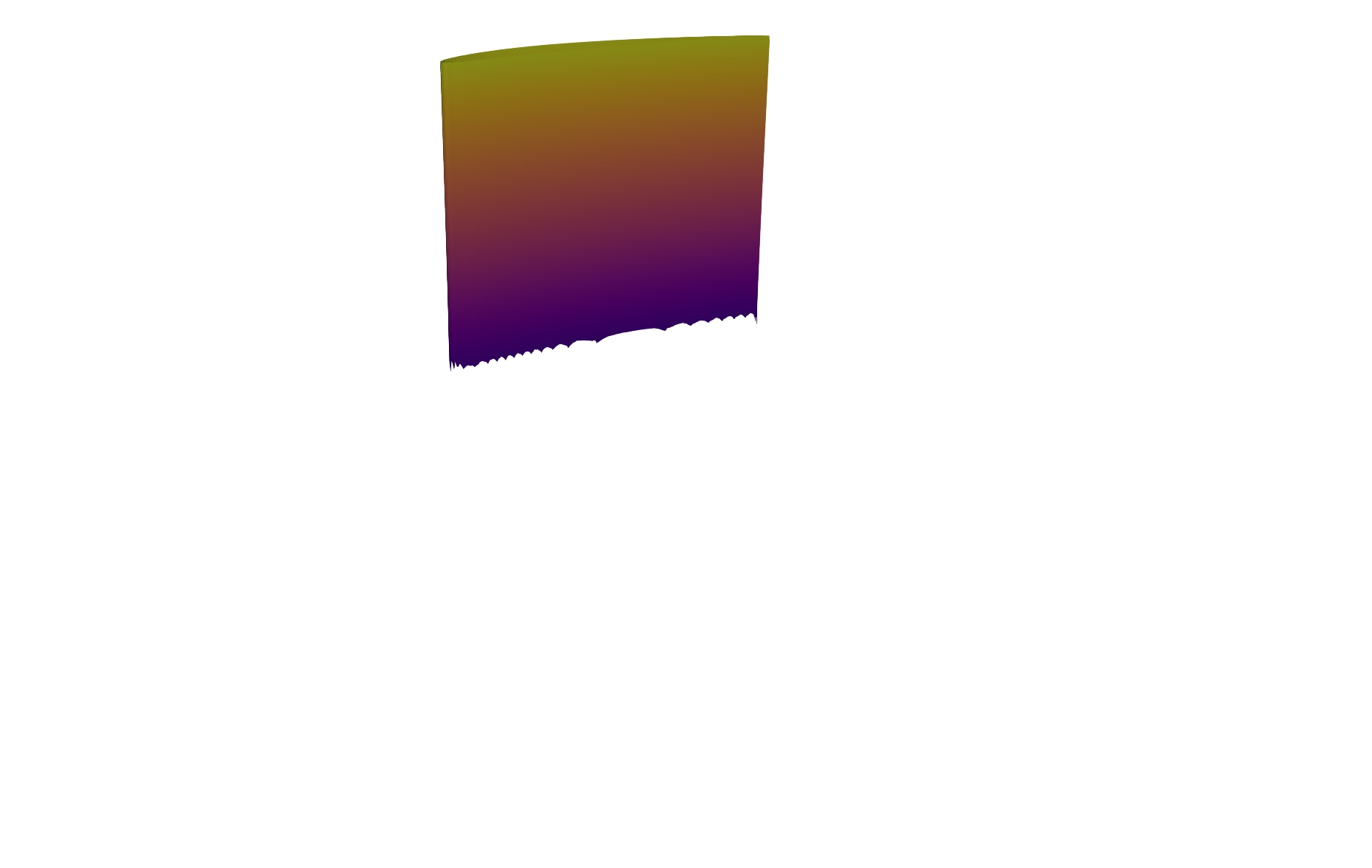}
  \includegraphics[width=0.3\textwidth]{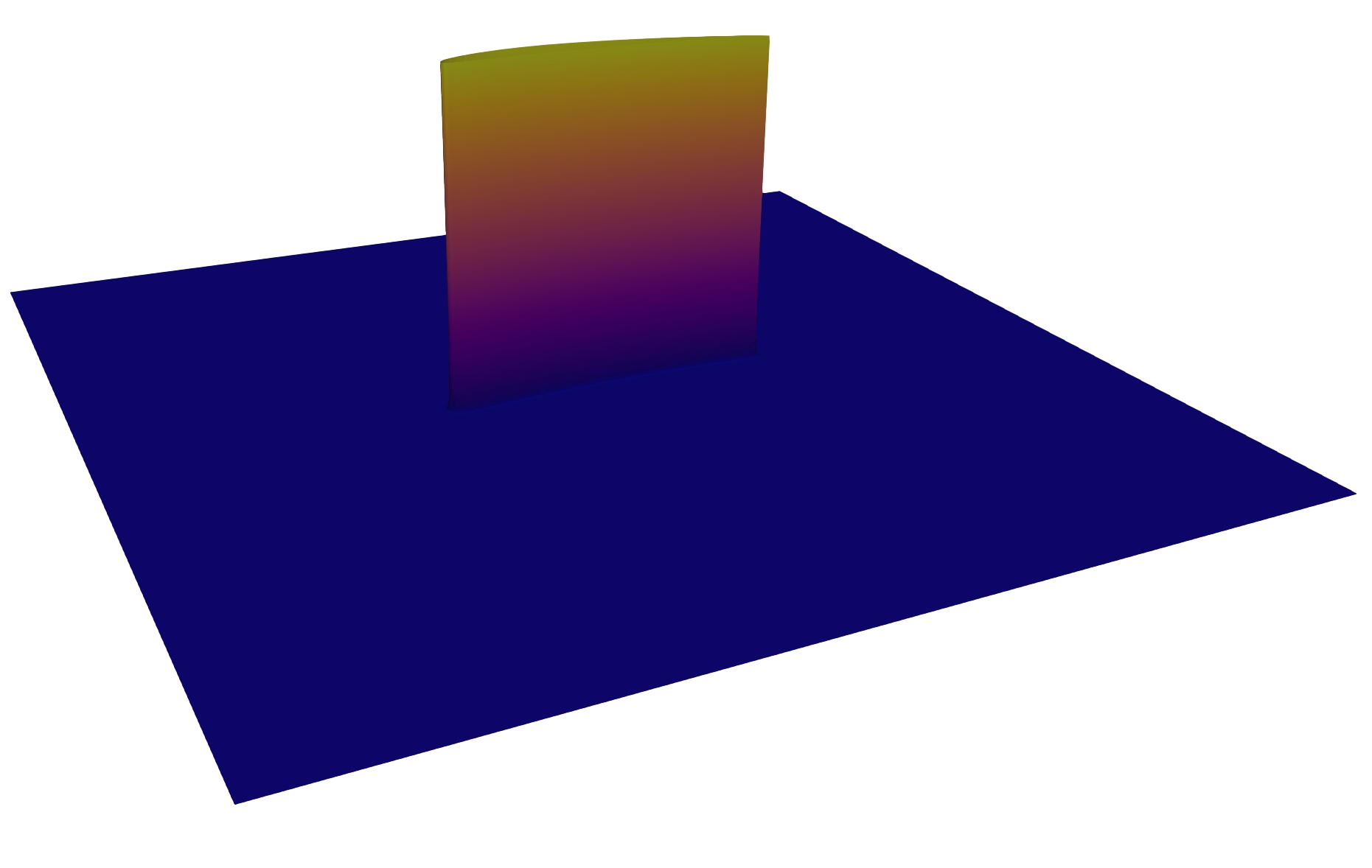} \\
  \caption{Solution of the boundary layer problem~\eqref{eq:blayer} for $k=0,1,2,3,4$. The left column shows the solution represented on the background mesh, the middle column shows the solution on the overlapping boundary-fitted mesh and the right column shows the composite multimesh solution.}
  \label{fig:blayer-3d}
\end{figure}

\begin{figure}[htbp]
  \centering
  \includegraphics[width=0.9\textwidth]{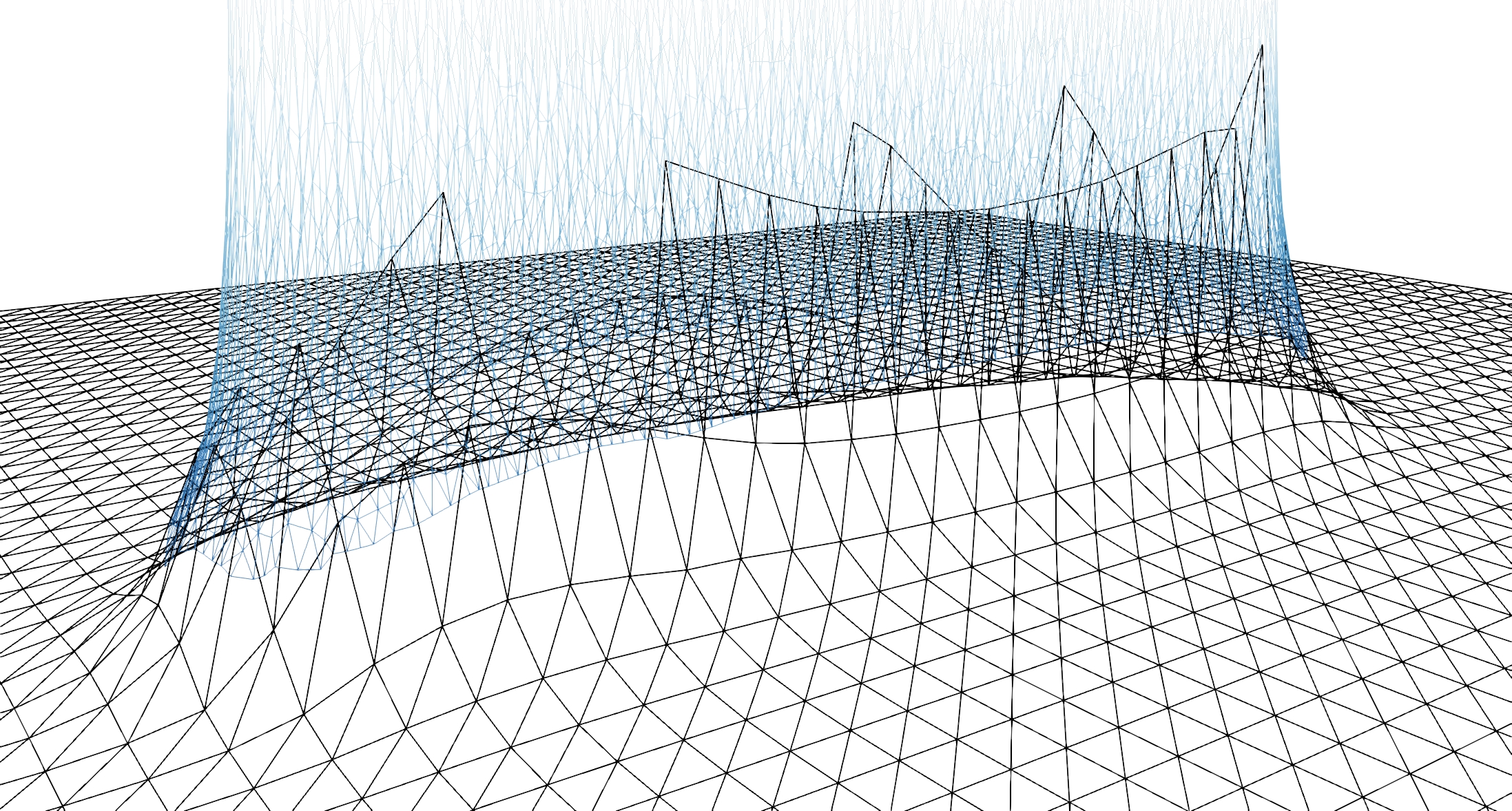} \\[1mm]
  \includegraphics[width=0.9\textwidth]{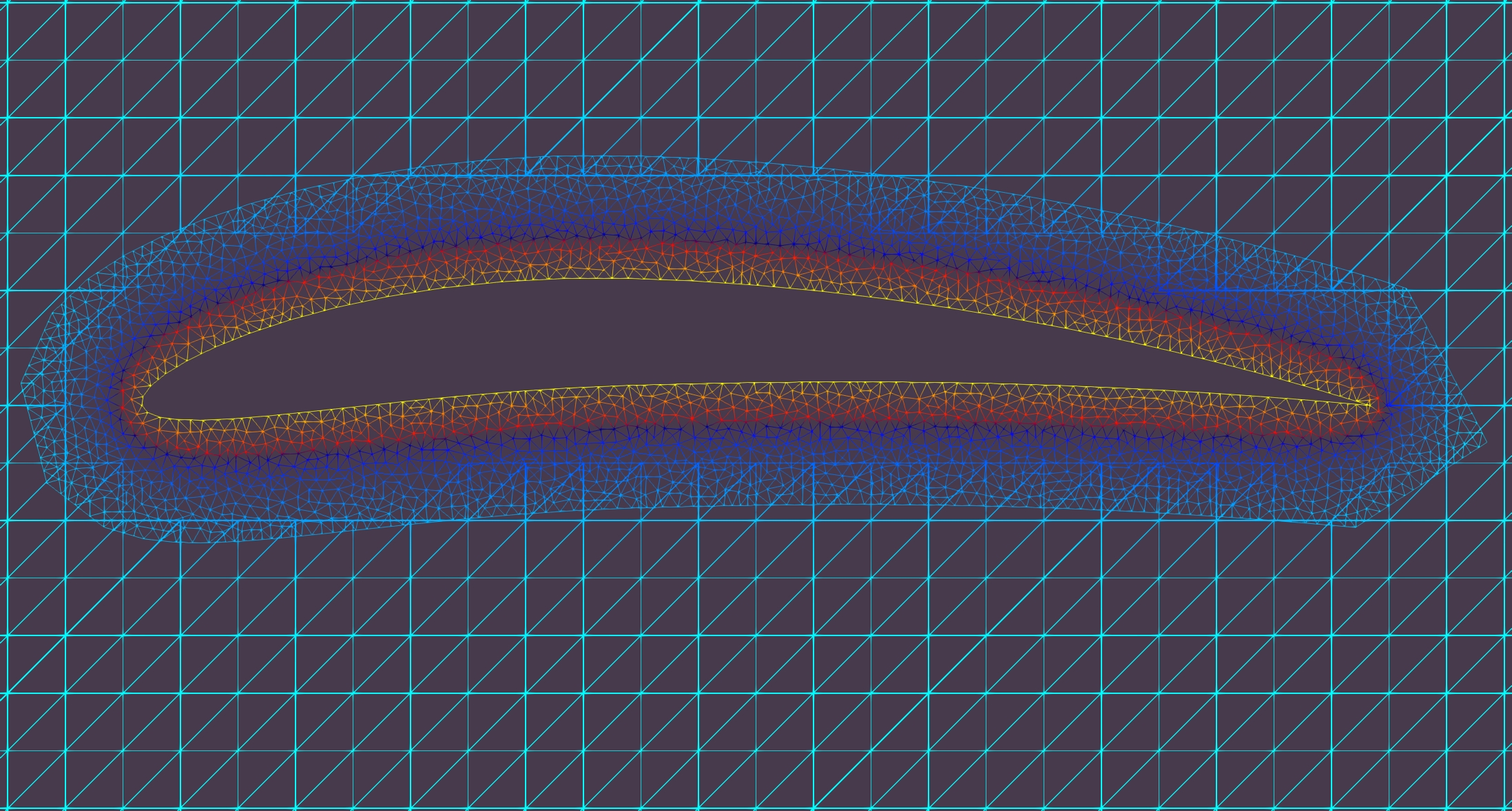} \\[1mm]
  \includegraphics[width=0.9\textwidth]{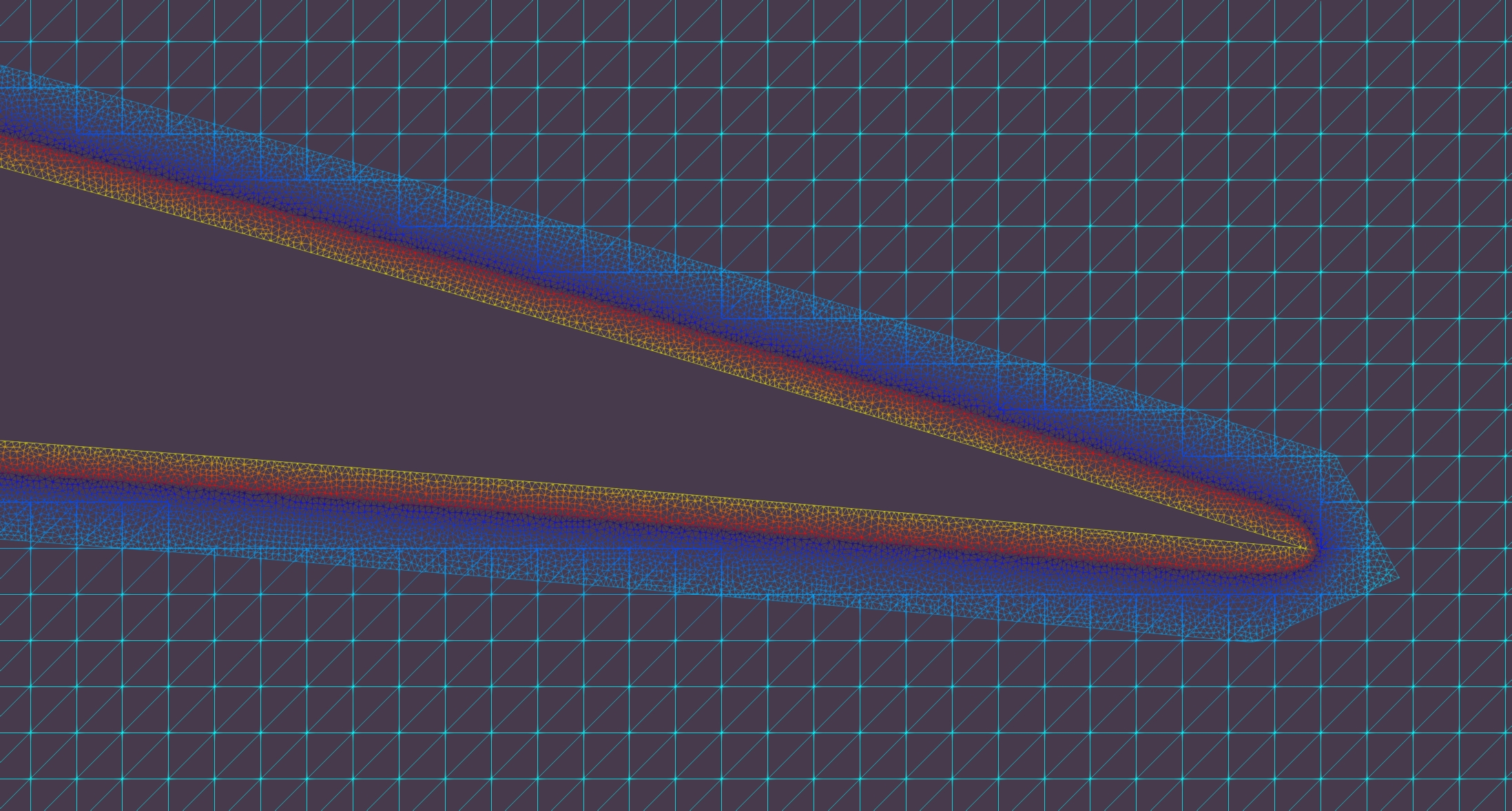}
  \caption{(Top) A 3D view of the matching solutions to the boundary layer problem~\eqref{eq:blayer} on the background mesh and the overlapping boundary-fitted mesh for $k = 0$. (Middle) The corresponding 2D view for $k = 0$. (Bottom) A detailed zoom close to the tip of the airfoil for the finest mesh ($k = 4$).}
  \label{fig:blayer-zoom}
\end{figure}

%---------------------------------------------------------------------------
\section{Conclusions}
\label{sec:conclusions}

We have analyzed a general framework for discretization of the Poisson equation posed on a domain defined by an arbitrary number of intersecting meshes with arbitrary mesh sizes. The analysis show that for sufficiently large Nitsche and stabilization parameters, the method is optimal and stable. As expected, there is a dependency on the maximum number of intersecting meshes in the coercivity, error analysis and in the condition number estimate. This was seen numerically in the accompanying paper~\cite{mmfem-1}, and here we are able to quantify this dependence. In addition, the numerical results presented in this paper show that the method is indeed stable when the meshes involved have vastly different mesh sizes.

As mentioned in the introduction, the multimesh method may be advantageous in the case of dynamic domains, since remeshing may be avoided. This is due to the fact that the computational geometry routines automatically identify the elements constituting the active meshes, and this can easily be done every time the domains move. Although so far only studied for two-dimensional problems~\cite{Dokken:2017aa} reports a speed up. The same approach is also applied in~\cite{DokkenNS}.

Future work involves extending the implementation to include three-dimensional meshes, which is a challenge due to requirements of efficient and accurate computational geometry routines in the case of arbitrary many intersecting meshes. That the multimesh formulation is valid in the case of two meshes in three dimensions is explored in~\cite{Johansson:2015aa} for the Stokes problem.

%---------------------------------------------------------------------------
\section{Acknowledgments}
\label{sec:acks}

August Johansson was supported by The Research Council of Norway through a Centres of Excellence grant to the Center for Biomedical Computing at Simula Research Laboratory, project number 179578, as well as by the Research Council of Norway through the FRIPRO Program at Simula Research Laboratory, project number 25123. Mats G.~Larson was supported in part by the Swedish Foundation for Strategic Research Grant No.~AM13-0029, the Swedish Research Council Grants Nos.~2013-4708, 2017-03911, and the Swedish Research Programme Essence. Anders Logg was supported by the Swedish Research Council Grant No.~2014-6093.

%---------------------------------------------------------------------------
\bibliographystyle{siamplain}
\bibliography{bibliography}

\end{document}